\documentclass[11pt]{amsart}
\usepackage{amssymb}
\usepackage{multirow}
\usepackage{mathtools}
\usepackage[pagebackref]{hyperref}
\hypersetup{breaklinks=true}
\usepackage{breakurl}
\usepackage{enumitem}

\usepackage[margin=1in]{geometry}

\newtheorem{thm}{Theorem}[section]
\newtheorem{cor}[thm]{Corollary}
\newtheorem{prop}[thm]{Proposition}
\newtheorem{lem}[thm]{Lemma}

\theoremstyle{definition}
\newtheorem{defn}[thm]{Definition}

\theoremstyle{remark}
\newtheorem{rem}[thm]{Remark}

\begin{document}

\title{Hodge Groups of Hodge Structures with Hodge Numbers $(n,0,\ldots,0,n)$}
\author{Laure Flapan}
\address{Mathematics Department, University of California, Los Angeles}
\email{lflapan@math.ucla.edu}

\maketitle

\begin{abstract}
We study the possible Hodge groups of simple polarizable $\mathbb{Q}$-Hodge structures with Hodge numbers $(n,0,\ldots,0,n)$. In particular, we generalize work of Moonen-Zarhin, Ribet, and Tankeev to completely determine the possible Hodge groups of such Hodge structures when $n$ is equal to $1$, $4$, or a prime $p$. In addition, we determine, under certain conditions on the endomorphism algebra, the possible Hodge groups when $n=2p$, for $p$ an odd prime. A consequence of these results is that, for all powers of a simple complex $2p$-dimensional abelian variety whose endomorphism algebra is of the specified type,  both the Hodge and General Hodge Conjectures hold. \end{abstract}
\maketitle
\tableofcontents{}

 \section{Introduction}
A $\mathbb{Q}$-Hodge structure of weight $w\ge1$ with Hodge numbers $(n,0,\ldots,0,n)$ is a $\mathbb{Q}$-vector space $V$ together with a decomposition into $n$-dimensional complex subspaces
$$V\otimes _\mathbb{Q}\mathbb{C}=V^{w,0}\oplus V^{0,w}$$
such that the two subspaces $V^{w,0}$ and  $V^{0,w}$ are conjugate to each other. Recently, Totaro \cite[Theorem 4.1]{totaro} classified all of the possible endomorphism algebras of simple polarizable $\mathbb{Q}$-Hodge structures with Hodge numbers $(n,0,\ldots,0,n)$. These possible endomorphism algebras fall into four broad types, under a classification by Albert of division algebras with positive involution \cite{albert}, which are referred to as Type I, Type II, Type III, and Type IV.

Complex abelian varieties of dimension $n$ are equivalent, up to isogeny, to polarizable $\mathbb{Q}$-Hodge structures of weight $1$ with Hodge numbers $(n,n)$. Thus, Totaro's result generalizes a result of Shimura's \cite[Theorem 5]{shimura}, which classifies all the possible endomorphism algebras of a complex abelian variety of fixed dimension.

If $V$ is a $\mathbb{Q}$-Hodge structure of even weight $w$ with decomposition
\[V\otimes _\mathbb{Q} \mathbb{C}=\bigoplus_{p+q=w}V^{p,q},\]
then the \emph{Hodge classes} of $V$ are the elements of $V$ that lie in the subspace $V^{w/2,w/2}$ in this decomposition. These Hodge classes are the subject of the long-standing Hodge Conjecture. The Hodge Conjecture states that if $X$ is a smooth projective variety, then, for any $p\ge 1$, all of the Hodge classes of $H^{2p}(X,\mathbb{Q})$ are $\mathbb{Q}$-linear combinations of rational cohomology classes of algebraic subvarieties of codimension $p$ in $X$.

For a rational sub-Hodge structure $W$ of a $\mathbb{Q}$-Hodge structure $V$, meaning a subspace $W\subset V$ such that the summands $W^{p,q}:=V^{p,q}\cap W\otimes_\mathbb{Q}\mathbb{C}$ endow $W$ with a $\mathbb{Q}$-Hodge structure, the \emph{level} of $W$ is
$l(W)=\max\{p-q \mid W^{p,q}\ne 0\}.$
The General Hodge Conjecture states that for a rational sub-Hodge structure $W\subset H^w(X,\mathbb{Q})$ such that $l(W)=w-2p$, there exists a Zariski-closed subset $Z$ of codimension $p$ in $X$ such that $W$ is contained in $\ker(H^w(X,\mathbb{Q})\rightarrow H^w(X-Z,\mathbb{Q}))$. 

The \emph{Hodge group} of a $\mathbb{Q}$-Hodge structure $V$ is a connected algebraic $\mathbb{Q}$-subgroup of $SL(V)$ whose invariants in the tensor algebra generated by $V$ and its dual $V^*$ are exactly the Hodge classes. Thus Hodge groups are objects of interest towards a better understanding of both the Hodge and General Hodge Conjectures. 

In this paper, we characterize the possible Hodge groups of simple polarizable $\mathbb{Q}$-Hodge structures with Hodge numbers $(n,0,\ldots,0,n)$ by applying Totaro's results about endomorphism algebras as well as by extending techniques of Moonen-Zarhin \cite{fourfold}, Ribet \cite{ribet1}, and Tankeev \cite{tankeevrib} for determining Hodge groups and combining these with more recent work of Green-Griffiths-Kerr \cite{domain} about domains of polarizable $\mathbb{Q}$-Hodge structures with specified Hodge group. In particular,  Proposition \ref{onedim}, Theorem \ref{primethm}, and Theorem \ref{fourfold} determine the possible Hodge groups when $n$ is equal to $1$, a prime $p$, or $4$, respectively. Moreover, for $n=2p$, where $p$ is an odd prime, Theorem \ref{2phodge} determines the possible Hodge groups when the endomorphism algebra is of Types I, II, or III as well as when the endomorphism algebra is of Type IV and has a particular type of action by an imaginary quadratic field. 

The results for $n=1$, $p$, and $4$ generalize known results about the possible Hodge groups of simple complex abelian varieties, while the results for $n=2p$ are new.

In particular, Ribet and Tankeev showed in \cite{ribet1} and \cite{tankeevrib}, respectively, that the Hodge group of a simple complex abelian variety $X$ of prime dimension is always equal to the Lefschetz group of $X$, meaning, roughly speaking, that the Hodge group of $X$ is always as large as possible. We show that, in fact, the above holds for all simple polarizable $\mathbb{Q}$-Hodge structures with Hodge numbers $(p,0,\ldots,0,p)$. Moreover, for simple polarizable $\mathbb{Q}$-Hodge structures with Hodge numbers $(2p,0,\ldots,0,2p)$, we establish conditions under which the above holds and conditions under which it does not.

Moonen and Zarhin addressed the case of simple complex abelian fourfolds in \cite{fourfold}.  They showed that the Hodge group of a simple complex abelian fourfold $X$ is equal to the Lefschetz group of $X$ except in the cases when the endomorphism algebra of $X$ is $\mathbb{Q}$ or a CM field of degree $2$ or $8$ over $\mathbb{Q}$. In all of these exceptional cases, an additional group is also possible. 

We show that for a simple polarizable $\mathbb{Q}$-Hodge structure $V$ with Hodge numbers $(4,0\ldots,0,4)$, an additional group other than the Lefschetz group can arise as the Hodge group in all the same exceptional cases that Moonen-Zarhin determined. 
When the endomorphism algebra of $V$ is a CM field, the additional possible group is analogous to the one found by Moonen and Zarhin.  In the case when the endomorphism algebra of $V$ is equal to $\mathbb{Q}$, Moonen and Zarhin used a construction of Mumford's \cite[Section 4]{mumford} to show that the additional group $SL(2)\times SO(4)$ arises. We show that the analogous group for an even-weight simple polarizable Hodge structure with Hodge numbers $(4,0,\ldots,0,4)$ does not arise, however that a different non-Lefschetz group, namely the group $SO(7)$ acting by the spin representation does arise.

The \emph{Hodge ring} of a smooth projective variety $X$ is defined by
\[\mathcal{B}^{\bullet}(X)=\bigoplus_{l\ge0} \left(H^{2l}(X,\mathbb{Q})\cap H^{l,l}\right).\]
Most of the proven cases of the Hodge Conjecture and General Hodge Conjecture, particularly for abelian varieties, are obtained by proving that the \emph{divisor ring} $\mathcal{D}^{\bullet}(X)$ of $X$, meaning the $\mathbb{Q}$-subalgebra of $\mathcal{B}^{\bullet}(X)$ generated by divisor classes, is equal to $\mathcal{B}^{\bullet}(X)$. Theorem \ref{2phodge} about simple polarizable $\mathbb{Q}$-Hodge structures with Hodge numbers $(2p,0,\ldots,0,2p)$ allows one to use this approach in Corollary \ref{genhg} to prove the Hodge and General Hodge Conjectures for all powers of a simple $2p$-dimensional abelian variety whose endomorphism algebra is of Type I or II in Albert's classification.

Mumford gave the first construction of a variety $X$ such that $\mathcal{B}^{\bullet}(X)\ne \mathcal{D}^{\bullet}(X)$ (see \cite{pohlmann}). Weil later observed that the exceptional Hodge classes in Mumford's example, meaning the Hodge classes that did not come from divisor classes, were still, in some sense, described by the endomorphisms of the rational cohomology of $X$  \cite{weil} (see Section \ref{hgconimp}). The exceptional Hodge classes satisfying the property Weil described came to be known an \emph{Weil classes}. We show in Corollary \ref{hgcon} that for any simple abelian variety $X$ of dimension $2p$ satisfying the hypotheses of Theorem \ref{2phodge}, the Hodge ring $\mathcal{B}^{\bullet}(X^k)$ of any power $X^k$ is generated by divisors and Weil classes. 

In order to further motivate the study of $\mathbb{Q}$-Hodge structures with Hodge numbers $(n,0,\ldots,0,n)$, let us say that a $\mathbb{Q}$-Hodge structure \emph{comes from geometry} if it is a summand of the rational cohomology of a smooth complex projective variety defined by a correspondence. Griffiths tranversality implies that a variation of Hodge structures of weight at least $2$ with no two adjacent non-zero Hodge numbers is locally constant \cite[Theorem 10.2]{voisin}. In particular, this applies to $\mathbb{Q}$-Hodge structures with Hodge numbers $(n,0,\ldots,0,n)$.  It follows that only countably many $\mathbb{Q}$-Hodge structures with Hodge numbers $(n,0,\ldots,0,n)$ and weight at least $2$ can come from geometry. 

Such Hodge structures arise naturally in geometry as the rational cohomology of an abelian variety or, for instance, of a rigid Calabi-Yau threefold. Moreover, we can produce a $\mathbb{Q}$-Hodge structure with Hodge numbers $(n,0,n)$ as follows. If $X$ is a smooth complex projective surface with maximal Picard number, namely the Picard number of $X$ is equal to $h^{1,1}(X)$, then $H^2(X,\mathbb{Q})$ modulo the subspace of Hodge classes is a $\mathbb{Q}$-Hodge structure with Hodge numbers $(p_g(X),0,p_g(X))$.  

As examples of higher dimensional constructions of such Hodge structures, in \cite{schreieder}, Schreieder shows how to construct a smooth complex projective variety with a prescribed set of Hodge numbers in a given degree $w$ (under suitable conditions on $h^{p,p}$ when $w$ is even). Arapura's shows in \cite{arapura} how to construct a $\mathbb{Q}$-Hodge structure with Hodge numbers $(n,0,\ldots,0,n)$ as a direct summand of the rational cohomology of a power $E^N$ of a CM elliptic curve. Bergeron-Millson-Moeglin show that $\mathbb{Q}$-Hodge structures with Hodge numbers $(n,0,\ldots,0,n)$ arise arithmetically as summands in the rational cohomology of Shimura varieties associated with a standard unitary group \cite[Corollary 6.2]{bergeron}. 

Aside form these concrete geometric examples, however, very little is known about the subset of the period domain of all $\mathbb{Q}$-Hodge structures with Hodge numbers $(n,0,\ldots,0,n)$ consisting of those that come from geometry.

The organization of the paper is as follows. In Section \ref{basics}, we introduce the necessary background on $\mathbb{Q}$-Hodge structures and their endomorphism algebras. Sections \ref{hgnumbers}-\ref{hgreps} focus on preliminary properties of Hodge structures with Hodge numbers $(n,0,\ldots,0,n)$. Sections \ref{TypeIsection}-\ref{typeiv} then focus on results about the Hodge groups of these Hodge structures depending on the endomorphism algebra type in Albert's classification. The main results of the paper occur in Section \ref{mainresults},  where the possible Hodge groups for $n=1$, $4$, $p$, and $2p$ are characterized. Section \ref{hgconimp} then gives consequences of the results in Section \ref{mainresults} in the context of the Hodge Conjecture and General Hodge Conjecture for simple abelian varieties of dimension $2p$.

\textbf{Acknowledgements.} I would like to thank my advisor Burt Totaro for his help and encouragement. I also thank Salman Abdulali, Radu Laza, and Ben Moonen for helpful conversations and comments. Finally, I thank the referee for helpful comments and suggestions. 

\section{Background on Hodge Structures}\label{basics}

\subsection{Hodge Structures and Polarizations}

A $\mathbb{Q}$-\emph{Hodge structure} $V$ is a finite dimensional $\mathbb{Q}$-vector space together with a decomposition into linear subspaces 
 $$V\otimes _\mathbb{Q} \mathbb{C}=\bigoplus_{(p,q)\in\mathbb{Z}^2} V^{p,q},$$
such that $\overline{V^{p,q}}=V^{q,p}$ and such that the \emph{weight grading} $p+q$ is defined over $\mathbb{Q}$.  Unless stated otherwise, the term ``Hodge structure" in this paper will always refer to a $\mathbb{Q}$-Hodge structure. A $\mathbb{Q}$-Hodge structure $V$ is of \emph{(pure) weight} $w$ if $V^{p,q}=0$ whenever $p+q$ is not equal to $w$.  For example, any smooth complex projective variety $X$ has a Hodge structure of (pure) weight $w$ on $H^w(X,\mathbb{Q})$.
 
 Alternatively, a $\mathbb{Q}$-Hodge structure can be defined as a finite dimensional $\mathbb{Q}$-vector space $V$ together with a homomorphism of $\mathbb{R}$-algebraic groups 
\begin{equation}\label{hgdef}h: R_{\mathbb{C}/\mathbb{R}}\mathbb{G}_m\rightarrow GL(V\otimes _\mathbb{Q}\mathbb{R}),\end{equation} where $R_{\mathbb{C}/\mathbb{R}}$ denotes the Weil restriction functor from $\mathbb{C}$ to $\mathbb{R}$. Here $h(z)$ acts on $V^{p,q}$ as multiplication by $z^{-p}\overline{z}^{-q}$. 
 
A $\mathbb{Q}$-Hodge structure $V$ has \emph{Hodge numbers} $(a_0,a_1,\ldots, a_w)$ if $V$ has weight $w\ge 0$ and
\begin{equation*}
\dim_\mathbb{C}V^{i,w-i}=
\begin{cases}
a_i &\mbox{for } 0\le i\le w\\
 0 &\mbox{otherwise}.
 \end{cases}
 \end{equation*}

A \emph{polarization} of a $\mathbb{Q}$-Hodge structure $V$ of weight $w$ is a bilinear form $\langle,\rangle: V\times V\rightarrow \mathbb{Q}$ that is alternating if $w$ is odd, symmetric if $w$ is even, and whose extension to $V\otimes _\mathbb{Q} \mathbb{C}$ satisfies:
\begin{enumerate}
\item $\langle V^{p,q}, V^{p',q'}\rangle=0$ if $p'\ne w-p$
\item $i^{p-q}(-1)^{\frac{w(w-1)}{2}}\langle x,\overline{x}\rangle >0$ for all nonzero $x\in V^{p,q}$.
\end{enumerate}

Note, for instance, that for $X$ a smooth complex projective variety, a choice of ample line bundle on $X$ will determine a polarization on the Hodge structure $H^w(X,\mathbb{Q})$. The category of polarizable $\mathbb{Q}$-Hodge structures is a semisimple abelian category \cite[Theorem 1.16]{moonenmt}.   All Hodge structures considered in this paper will be polarizable.

\subsection{The Endomorphism Algebra}\label{endalgsection}

Let $V$ be a simple $\mathbb{Q}$-Hodge structure with polarization $\langle,\rangle: V\times V\rightarrow \mathbb{Q}$. Then the endomorphism algebra $L=\mathrm{End}_{\mathbb{Q}-\mathrm{HS}}(V)$ of $V$ as a $\mathbb{Q}$-Hodge structure is a division algebra over $\mathbb{Q}$ with an involution $a\to \overline{a}$, given by $\langle ax,y\rangle=\langle x, \overline{a}y\rangle$ for all $x,y \in V$. This involution $^{-}$ is called the \emph{Rosati involution}. 

The Rosati involution is a \emph{positive involution}, meaning that, if $\Sigma(L)$ is the set of embeddings of the endomorphism algebra $L$ into $\mathbb{C}$, then for every element $\sigma\in \Sigma(L)$, the reduced trace $\sigma(\mathrm{tr}^L_\mathbb{Q}(x\overline{x}))$ is positive as an element of $\mathbb{R}$ for all nonzero $x\in V$ \cite[Remark 1.20]{moonenmt}. It follows that if $L$ is a field, then $L$ is either a totally real or a CM field, where a \emph{CM field} means a totally imaginary quadratic extension of a totally real number field. Namely, if $L$ is a field, then the Rosati involution on $L$ just corresponds to complex conjugation.
 
 Letting $F_0$ be the center of $L$ and $F$ the subfield of $F_0$ fixed by the Rosati involution, Albert's classification of division algebras over a number field that have positive involution \cite[Chapter X, \S11]{albert} yields that $L$ is one of the following four types:
\begin{enumerate}
\item Type I: $L=F$ is totally real

\item Type II: $L$ is a totally indefinite quaternion algebra over the totally real field $F$

\item Type III: $L$ is a totally definite quaternion algebra over  the totally real field $F$

\item Type IV: $L$ is a central simple algebra over the CM field $F_0$.
\end{enumerate}

\subsection{The Mumford-Tate Group and Hodge Group}

The \emph{Mumford-Tate group} $MT(V)$ of a polarizable $\mathbb{Q}$-Hodge structure $V$ is the $\mathbb{Q}$-Zariski closure of the homomorphism 
$h: R_{\mathbb{C}/\mathbb{R}}\mathbb{G}_m\rightarrow GL(V\otimes_\mathbb{Q}\mathbb{R})$ which defines the Hodge structure on $V$. Note that the Mumford-Tate group is thus a connected group. Define the cocharacter 
$$\mu: \mathbb{G}_m\rightarrow R_{\mathbb{C}/\mathbb{R}}\mathbb{G}_m$$
 to be the unique cocharacter such that $z\circ \mu$ is the identity in $\mathrm{End(\mathbb{G}_m)}$ and $\overline{z}\circ \mu$ is trivial. Then the Mumford-Tate group of $V$ may be alternatively described as the smallest $\mathbb{Q}$-algebraic group contained in $GL(V)$ such that 
 $$h\circ \mu: \mathbb{G}_{m,\mathbb{C}}\rightarrow GL(V\otimes_\mathbb{Q}\mathbb{C})$$
 factors through $MT(V)_\mathbb{C}$.

Instead of working with the Mumford-Tate group, we will generally work with a slightly smaller connected group, called the \emph{Hodge group} $Hg(V)$ of $V$. The Hodge group is the $\mathbb{Q}$-Zariski closure of the restriction of the homomorphism $h$ to the circle group 
$$U_1=\ker(\mathrm{Norm}: R_{\mathbb{C}/\mathbb{R}}\mathbb{G}_m\rightarrow \mathbb{G}_m).$$

If the Hodge structure $V$ is of weight $0$, then $Hg(V)$ and $MT(V)$ coincide. If $V$ is of nonzero weight, then $MT(V)$ contains $\mathbb{G}_m$, and in fact, is equal to the almost direct product  $\mathbb{G}_m\cdot Hg(V)$ in $GL(V)$.
 
The category of polarizable $\mathbb{Q}$-Hodge structures is a semi-simple Tannakian category, which, in particular, implies that $MT(V)$ and $Hg(V)$ are reductive $\mathbb{Q}$-groups \cite[Theorem 1.16]{moonenmt}.

\begin{rem}\label{basicprop} The key property of the Hodge group $Hg(V)$ of a polarizable $\mathbb{Q}$-Hodge structure $V$ is that $Hg(V)$ is the subgroup of $SL(V)$ preserving the Hodge classes in all the $\mathbb{Q}$-Hodge structures $T$ formed as  finite direct sums of spaces of the form $T^{k,l}:=V^{\otimes k}\otimes (V^*)^{\otimes l}$, where $V^*$ denotes the dual of the vector space $V$ \cite[I.B.1]{domain}. 
\end{rem}

\begin{rem}\label{triv} If $V$ is a polarizable $\mathbb{Q}$-Hodge structure with Hodge group $\{1\}$, then $h(z)$ acts as the identity on nonzero $V^{p,q}$ in the decomposition of $V\otimes _{\mathbb{Q}}\mathbb{C}$. However, by definition $h(z)$ acts on $V^{p,q}$ as multiplication by $z^{-p}\overline{z}^{-q}$. Hence, if $Hg(V)=\{1\}$, then $V^{p,q}$ is zero for all $p,q$ such that $p\ne q$. 
\end{rem}

\begin{rem}\label{ss} If the endomorphism algebra $L$ of a polarizable $\mathbb{Q}$-Hodge structure $V$ has no simple factors of Type IV, then the Hodge group of $V$ is semisimple \cite[Proposition 1.24]{moonenmt}.
\end{rem}

\subsection{The Lefschetz Group}

Since elements of the endomorphism algebra $L=\mathrm{End}_{\mathbb{Q}-\mathrm{HS}}(V)$ of a polarizable $\mathbb{Q}$-Hodge structure $V$ preserve the Hodge decomposition, elements of $L$ may be viewed as Hodge classes of the $\mathbb{Q}$-Hodge structure $\mathrm{End}_\mathbb{Q}(V)\cong V\otimes V^*$. Remark \ref{basicprop} yields that these Hodge classes are the elements of $\mathrm{End}_\mathbb{Q}(V)$ which are invariant under the action of the Hodge group $Hg(V)$. Namely we have
\[L=[\mathrm{End}_\mathbb{Q}(V)]^{Hg(V)}.\]
In particular, the Hodge group of $V$ is contained in the connected component of the centralizer of $L$ in the $\mathbb{Q}$-group $SL(V)$.

Let the \emph{Lefschetz group} $Lef(V)$ of $V$ be the connected component of the centralizer of $L$ in:
\begin{equation*}
\begin{cases}
Sp(V) &\mbox{if $V$ is of odd weight}\\
SO(V) &\mbox{if $V$ is of even weight}.
\end{cases}
\end{equation*}
It should be noted that the definition of the Lefschetz group used here corresponds with the definition used by Murty in \cite[Section 2]{murty2}, which is the connected component of the identity in the definition used by Murty in \cite[Section 3.6.2]{murty}.

\begin{rem}\label{lefrem}
If $\langle,\rangle$ is a polarization on $V$, then $Hg(V)$ preserves $\langle,\rangle$, which yields the inclusions
\[Hg(V)\subseteq Lef(V).\]
\end{rem}

\section{Hodge Structures with Hodge Numbers $(n,0,\ldots,0,n)$}\label{hgnumbers}

We now specify our discussion to the main subject of this paper, namely polarizable $\mathbb{Q}$-Hodge structures with Hodge numbers $(n,0,\ldots,0,n)$. The first crucial observation to make about these Hodge structures is the following.

\begin{rem}\label{equiv}There is an equivalence of categories between the category of $\mathbb{Q}$-Hodge structures of weight $w$ and Hodge numbers $(n,0\ldots,0,n)$, and the category of $\mathbb{Q}$-Hodge structures of weight $1$ and Hodge numbers $(n,n)$. This equivalence is given simply by identifying $V^{w,0}\subset V \otimes _\mathbb{Q} \mathbb{C}$ with $V^{1,0}$. When $w$ is odd, this equivalence preserves polarizability. However, this is not the case when $w$ is even.  
\end{rem}

However, there is also an equivalence of categories between the category of complex abelian varieties up to isogeny and the category of polarizable $\mathbb{Q}$-Hodge structures with Hodge numbers $(n,n)$ given by identifying a complex abelian variety with its weight-$1$ rational cohomology. Thus when analyzing properties of polarizable $\mathbb{Q}$-Hodge structures with Hodge numbers $(n,0,\ldots,0,n)$ in the context, for instance, of their endomorphism algebras, Hodge groups, or Lefschetz groups, we may use existing machinery about complex abelian varieties to deal with the odd-weight case. For the even weight case, however, new techniques are needed. 

In \cite{shimura}, Shimura classifies all of the possible endomorphism algebras of simple polarized complex abelian varieties, which in light of the above equivalence, yields a classification of of the possible endomorphism algebras of odd-weight simple polarizable Hodge structures with Hodge numbers $(n,0,\ldots,0,n)$. Totaro completes this classification in \cite{totaro} to include all even-weight such Hodge structures as well. For reference, we include this classification below.

\subsection{Endomorphism Algebra Classification}\label{endclass}
Let $V$ be a simple polarizable $\mathbb{Q}$-Hodge structure of weight $w\ge 1$ with Hodge numbers $(n,0,\ldots,0,n)$ and $L$ its endomorphism algebra. Let $F_0$ be the center of $L$, so that $F_0$ is a totally real or CM field (see Section \ref{endalgsection}), and let $F$ be the maximal totally real subfield of $F_0$. Writing $g=[F:\mathbb{Q}]$, $2n=m[L:\mathbb{Q}]$, $q^2=[L:F_0]$,  let $B\cong M_m(L^{\mathrm{op}})$ be the centralizer of $L$ in $\mathrm{End}_\mathbb{Q}(V)$. 

If $L$ is of Type IV in Albert's classification, so that $F_0$ is a CM field, let $\Sigma(F_0)=\{\sigma_1,\ldots,\sigma_g,\overline{\sigma}_1,\ldots,\overline{\sigma}_g\}$ be the set of embeddings of $F_0$ into $\mathbb{C}$. Then $L\otimes _\mathbb{Q}\mathbb{C}$ is isomorphic to $2g$ copies of $M_q(\mathbb{C})$, one for each embedding $\sigma_i\in \Sigma(F_0)$. This decomposition of $L\otimes _\mathbb{Q}\mathbb{C}$ yields a decomposition of $V^{w,0}\subset V\otimes_\mathbb{Q}\mathbb{C}$ into summands $V^{w,0}(\sigma_i)$ on which $F_0$ acts via the embedding $\sigma_i \in \Sigma(F_0)$.  Letting $n_{\sigma_i}=\dim_\mathbb{C}V^{w,0}(\sigma_i)$, we then have $n_{\sigma_i}+n_{\overline{\sigma}_i}=mq$ for each $i=1,\ldots,g$.

\begin{thm}\label{totthm}(Totaro \cite[Theorem 4.1] {totaro}) In the above notation, if $V$ is a simple polarizable $\mathbb{Q}$-Hodge structure  of weight $w\ge 1$ with Hodge numbers $(n,0,\ldots,0,n)$, then $[L:\mathbb{Q}]$ divides $2n$ and $[F:\mathbb{Q}]$ divides $n$. 

Conversely, every division algebra with positive involution satisfying these two bounds is the endomorphism algebra of some simple polarizable $\mathbb{Q}$-Hodge structure with Hodge numbers $(n,0,\ldots,0,n)$ except for the following $5$ odd-weight and $7$ even-weight exceptional cases.

Odd-weight exceptional cases:
\begin{enumerate}
\item Type III and $m=1$
\item Type III, $m=2$, $\mathrm{disc}(B,^{-})=1$ in $F^*/(F^*)^2$
\item Type IV and $\sum_{i=1}^g n_{\sigma_i}n_{\overline{\sigma}_i}=0$, unless $m=q=1$
\item Type IV, $m=2$, $q=1$, and $n_{\sigma_i}=n_{\overline{\sigma}_i}=1$ for all $i=1,\ldots,g$
\item Type IV, $m=1$, $q=2$, and $n_{\sigma_i}=n_{\overline{\sigma}_i}=1$ for all $i=1,\ldots,g$
\end{enumerate}

Even-weight exceptional cases
\begin{enumerate}
\item Type II and $m=1$
\item Type II, $m=2$, $\mathrm{disc}(B,^{-})=1$ in $F^*/(F^*)^2$
\item Type IV and $\sum_{i=1}^g n_{\sigma_i}n_{\overline{\sigma}_i}=0$, unless $m=q=1$
\item Type IV, $m=2$, $q=1$, and $n_{\sigma_i}=n_{\overline{\sigma}_i}=1$ for all $i=1,\ldots,g$
\item Type IV, $m=1$, $q=2$, and $n_{\sigma_i}=n_{\overline{\sigma}_i}=1$ for all $i=1,\ldots,g$
\item Type I and $m=2$
\item Type I, $m=4$, and $(V,\langle,\rangle)$ has discriminant $1$ in $F^*/(F^*)^2$
\end{enumerate}
\end{thm}

\section{The Lefschetz Group of a $\mathbb{Q}$-Hodge Structure with Hodge Numbers $(n,0,\ldots,0,n)$}\label{lefschetz}
Let $A$ be a central simple algebra over a field $K_0$ with involution $^{-}$. Let $K$ be the subfield of $K_0$ fixed by the involution $^{-}$. Define
$$\mathrm{Sym}(A,^{-})=\{f\in A \mid \overline{f}=f\}$$
$$\mathrm{Alt}(A,^{-})=\{f\in A \mid \overline{f}=-f\}.$$

Then, following the Book of Involutions \cite{boi} and writing $[A:K_0]=q^2$, we say that the involution $^{-}$  on $A$ is \emph{orthogonal} if $K_0=K$ and $\dim_K\mathrm{Sym}(A,^{-})=\frac{q(q+1)}{2}$. We say the involution is \emph{symplectic} if $K_0=K$ and $\dim_K\mathrm{Sym}(A,^{-})=\frac{q(q-1)}{2}$. Finally, we say the involution is \emph{unitary} if $K_0\ne K$.  There are no other possibilities for the involution $^{-}$ \cite[Proposition I.2.6]{boi}. 

Letting $A^*$ denote the group of invertible elements of $A$, define the group of \emph{isometries} of $A$ by
$$\mathrm{Iso}(A,^{-})=\{g\in A^*\mid \overline{g}=g^{-1}\}.$$
Then write
$$\mathrm{Iso}(A,^{-})=
\begin{cases} O(A,^{-}) &\mbox{if } ^{-} \text{ is orthogonal}\\
Sp(A,^{-}) &\mbox{if } ^{-} \text{ is symplectic}\\
U(A,^{-}) &\mbox{if } ^{-} \text{ is unitary}.
\end{cases}$$

If the involution $^{-}$ is orthogonal, then, as an algebraic group, the kernel $O^+(A,^{-})$ of the reduced norm map $O(A, ^{-})\rightarrow \{\pm1\}$ is a  $K$-form of $SO(q)$, meaning the two groups are isomorphic over an algebraic closure of $K$. If the involution $^{-}$ is symplectic, then $Sp(A,^{-})$ is a $K$-form of $Sp(q)$ (in this case $q$ must be even). If the involution $^{-}$ is unitary, then $U(A,^{-})$ is a $K$-form of $GL(q)$  and $SU(A,^{-}):=\ker(\mathrm{Norm}_{A/K}\colon U(A,^{-}) \rightarrow \mathbb{G}_{m,K})$ is a $K$-form of $SL(q)$.

Let $V$ be a simple polarizable $\mathbb{Q}$-Hodge structure $V$ with Hodge numbers $(n,0,\ldots,0,n)$ and let $L$ be its endomorphism algebra. Using the notation introduced above as well as the notation of Section \ref{endclass}, we now list in Table \ref{exphg} Lefschetz group of $V$ depending on the type that $L$ has in Albert's classification. Note that in the table, the notation $_FV$ denotes $V$ considered as an $F$-vector space.

\begin{table}[h]
\caption{The Lefschetz Group of $V$ Depending on $L=\mathrm{End}_{\mathbb{Q}-HS}(V)$}
\label{exphg}
\centering
\begin{tabular}{|c|c|c|}
\hline
 \raisebox{0pt}[30pt][0pt]{$L$} & \multicolumn{2}{|c|}{Lef(V)}\\[0.5ex]
\cline{2-3}

& Odd Weight & Even Weight\\[0.5ex]
\hline
\hline
Type I & $R_{F/\mathbb{Q}}Sp(_FV)$ & $R_{F/\mathbb{Q}}SO(_FV)$\\[0.5ex]
\hline
Type II & $R_{F/\mathbb{Q}}Sp(B,^{-})$&$R_{F/\mathbb{Q}}O^+(B,^{-})$\\[0.5ex]
\hline
Type III & $R_{F/\mathbb{Q}}O^+(B,^{-})$&$R_{F/\mathbb{Q}}Sp(B,^{-})$\\[0.5ex]
\hline
Type IV & $R_{F/\mathbb{Q}}U(B,^{-})$&$R_{F/\mathbb{Q}}U(B,^{-})$\\[0.5ex]
\hline

\end{tabular}
\end{table}

\section{Irreducible Representations of the Hodge Group}\label{section5}

Let $V$ be a simple polarizable $\mathbb{Q}$-Hodge structure of weight $w\ge1$ with Hodge numbers $(n,0,\ldots, 0,n)$ and endomorphism algebra $L$. Let $H=Hg(V)$ be the Hodge group of $V$.

If $L$ is of Types I, II, or III in Albert's classification, then $F_0=F$ and so, letting $\Sigma(F)$ denote the set of embeddings of $F$ into $\mathbb{C}$, the action of $L\otimes_\mathbb{Q}\mathbb{C}$ on $V_\mathbb{C}:=V\otimes _\mathbb{Q}\mathbb{C}$ induces a decomposition as an $H_\mathbb{C}$-representation
\begin{equation}\label{DL0}V_\mathbb{C}= \bigoplus_{\sigma\in \Sigma(F)} qW_{\sigma}.\end{equation}

If $L$ is of Type IV, then $F_0$ is a CM field with maximal totally real subfield $F$ and $V_\mathbb{C}$ decomposes as an $H_\mathbb{C}$-representation as
\begin{equation}\label{DL1}V_\mathbb{C}= \bigoplus_{\sigma\in \Sigma(F)} \left(qW_{\sigma}\oplus qW_{\sigma}^*\right).\end{equation}

In both cases, each of the summands $W_\sigma$ for $\sigma \in  \Sigma(F)$ has complex dimension $mq$.

\begin{lem}\label{irreducible}
Each representation $W_\sigma$ of the group $H_\mathbb{C}$ is irreducible.
\end{lem}
\begin{proof}
By construction each representation $W_\sigma$ is nonzero and hence $[W_\sigma\otimes W_\sigma^*]^{H_\mathbb{C}}$ is nonzero. Namely in the decomposition of $[V_\mathbb{C}\otimes V_\mathbb{C}^*]^{H_\mathbb{C}}$ coming from (\ref{DL0}) respectively (\ref{DL1}), the term 
\[\left[\bigoplus_{\sigma\in \Sigma(F)}q^2(W_\sigma\otimes W_\sigma^*)\right]^{H_\mathbb{C}} \ \ \text{ respectively } \left[\bigoplus_{\sigma\in \Sigma(F)}2q^2(W_\sigma\otimes W_\sigma^*)\right]^{H_\mathbb{C}},\]
has dimension at least $q^2[F_0:\mathbb{Q}]=[L:\mathbb{Q}]$. But since $L$ is the endomorphism algebra of $V$, we know $[V_\mathbb{C}\otimes V_\mathbb{C}^*]^{H_\mathbb{C}}$ has dimension exactly equal to $[L:\mathbb{Q}]$. Therefore, each term $[W_\sigma\otimes W_\sigma^*]^{H_\mathbb{C}}$ has dimension equal to $1$, meaning that the representation of $H_\mathbb{C}$ on $W_\sigma$ is irreducible. 
\end{proof}


The group $H_\mathbb{C}$, up to permutation of factors, has a canonical decomposition as an almost direct product of its center $Z(H)_\mathbb{C}$ and its simple factors $H_i$ given by
\[H_\mathbb{C}=Z(H)_\mathbb{C}\cdot H_1\cdots H_s.\]
Passing to Lie algebras and writing $\mathfrak{c}=\mathrm{Lie}(Z(H))_\mathbb{C}$ and $\mathfrak{g}_i=\mathrm{Lie}(H_i)$ yields
\[\mathrm{Lie}(H)_\mathbb{C}=\mathfrak{c}\times \mathfrak{g}_1\times \cdots \times \mathfrak{g}_s.\]

Hence if $W_\sigma\subset V_\mathbb{C}$ is one of the irreducible $H_\mathbb{C}$-submodules introduced above, we have a decomposition of $W_\sigma$ as a representation of $\mathrm{Lie}(H)_\mathbb{C}$ given by
\begin{equation}\label{decompW} W_\sigma=\chi\boxtimes \rho_1\boxtimes\cdots \boxtimes \rho_s,\end{equation}
where $\chi$ is a character of $\mathfrak{c}$ and the $\rho_i$ are irreducible representations of the simple factors $\mathfrak{g}_i$. Note that in this notation the $\rho_i$ are allowed to be trivial.

In order to study this irreducible representation $W_\sigma$, we first recall some facts and terminology about representations of semisimple Lie algebras.

\subsection{Representations of Semisimple Lie Algebras}

Suppose $\mathfrak{g}$ is a semisimple Lie algebra over an algebraically closed field $K$ and let $\mathfrak{h}$ be a Cartan subalgebra of $\mathfrak{g}$.  Let $R$ be the root system of $\mathfrak{g}$ with respect to $\mathfrak{h}$ and let $B=\{\alpha_1,\ldots, \alpha_l\}$ be a set of simple roots of $R$ with corresponding set of coroots $B^{\vee}=\{\alpha^{\vee}\mid \alpha\in B\}$.  If $w_0$ is the longest element of the Weyl group of $R$ with respect to the basis $B$, let $\lambda\to \lambda':=-w_0(\lambda)$ denote the opposition involution on $\mathfrak{h}^*$. If $\lambda$ is a dominant weight, we can write 
$$\lambda=\sum_{\alpha \in B} c_\alpha \cdot \alpha,$$
where all the $c_\alpha$ are in $\mathbb{Q}_{\ge0}$. By Lemma 3.3 in \cite{moonenmt}, we have that $c_\alpha + c_{\alpha'}$ lies in $\mathbb{Z}_{\ge0}$ for all $\alpha \in B$. Then define
\[\mathrm{length}(\lambda)=\min_{\alpha\in B}c_\alpha + c_{\alpha'},\]

If $R$ is an irreducible root system, we say that a dominant weight $\lambda$ is \emph{minuscule} if $\langle \lambda,\alpha^{\vee}\rangle$ lies in $\{-1,0,1\}$ for all $\alpha\in R$ \cite[Chapter VIII, \S 7.3]{bourbaki}. Then $\mathrm{length}(\lambda)$ is equal to 1 if and only if $\lambda$ is a minuscule weight and $R$ is of classical type, meaning of type $A_l$, $B_l$, $C_l$, or $D_l$ \cite[Example 3.6]{moonenmt}.

We make extensive use of Table \ref{selfdual}, reproduced from \cite{moonenmt}. The table describes for a given root system with minuscule weight $\lambda$, the corresponding representation $V(\lambda)$ (omitting the weight 0).  The table gives the dimension and autoduality of the representation $V(\lambda)$.  Note that in the table, the symbol $-$ denotes a symplectic representation, the symbol $+$ denotes an orthogonal representation, and $0$ denotes a non-self-dual representation.

\begin{table}
\caption{Minuscule weights in irreducible root systems}
\label{selfdual}
\centering
\begin{tabular}{|c|c|c|c|c|}
\hline
Root system & Minuscule weight & Representation & Dimension & Autoduality\\ [0.5ex]
\hline
$A_{l}$ & $\overline{\omega}_j (1\le j\le l)$ & $\wedge^j(\mathrm{Standard})$ & \raisebox{0pt}[15pt][0pt]{${l+1}\choose {j}$} & $(-1)^j$ \text{ if } $l=2j-1$\\
& & & & $0$ \text{ otherwise}\\
\hline
$B_{l}$& $\overline{\omega}_l$ & $\mathrm{Spin}$ & \raisebox{0pt}[12pt]{$2^l$} & $+$ \text{ if } $l\equiv 0,3\mod4$\\
&&&& $-$ \text{ if }$l\equiv 1,2\mod4$\\
\hline
$C_l$ & $\overline{\omega}_1$ & $\mathrm{Standard}$ & \raisebox{0pt}[12pt]{$2l$} & $-$\\
\hline
$D_l$ & $\overline{\omega}_1$ & $\mathrm{Standard}$ & \raisebox{0pt}[12pt]{$2l$} & $+$\\
& $\overline{\omega}_{l-1},\ \overline{\omega}_l$ & $\mathrm{Spin}^-, \text{ resp. } \mathrm{Spin}^+$ & \raisebox{0pt}[12pt]{$2^{l-1}$}& $+$ \text{ if } $l\equiv 0\mod4$\\
&&&&$-$ \text{ if } $l\equiv 2 \mod4$\\
&&&&$0$ \text{ if } $l\equiv 1 \mod2$\\
\hline
$E_6$ & $\overline{\omega}_1$ & & \raisebox{0pt}[12pt]{27}& $0$\\
&$\overline{\omega}_6$ & & \raisebox{0pt}[12pt]{27} & $0$\\
\hline
$E_7$ & $\overline{\omega}_7$ & & \raisebox{0pt}[12pt]{56} & $-1$\\
\hline

\end{tabular}
\end{table}

\subsection{The Representation $W_\sigma$}

We now return to our consideration of the $mq$-dimensional irreducible representation $W_\sigma=\chi\boxtimes \rho_1\boxtimes\cdots \boxtimes \rho_s$ of the complexified Hodge group $H_\mathbb{C}$. 

\begin{lem}\label{tablelem} Each nontrivial $\rho_i$ for $1\le i \le s$  in the above notation satisfies\begin{enumerate}
\item \label{TL1}The highest weight of $\rho_i$ is minuscule and $\mathfrak{g}_i$ is of classical type
\item \label{TL2}If $\rho_i$ is self dual, then $\rho_i$ is even-dimensional
\item \label{TL3}If $\rho_i$ is symplectic with $\dim_\mathbb{C}(\rho_i) \equiv 2 \pmod{4}$, then $\mathfrak{g}_i$ is of type $C_l$, where $l\ge1$ is odd, and the representation is the standard representation of $\mathfrak{sp}_{2l}$
\item \label{TL4}If $\rho_i$ is orthogonal with $\dim_\mathbb{C}(\rho_i) \equiv 2 \pmod{4}$, then $\mathfrak{g}_i$ is of type either:
\begin{enumerate}
\item \label{TL5}$D_l$, where $l\ge1$ is odd, and the representation is the standard representation of $\mathfrak{so}_{2l}$
\item \label{binom} $A_{2^k-1}$, where $k\ge3$, and the representation is $\bigwedge^{2^{k-1}}(\mathrm{St})$, where $\mathrm{St}$ denotes the standard representation of $\mathfrak{sl}_{2^k}$
\end{enumerate}
\end{enumerate}
\end{lem}

\begin{proof}
Moonen proves in \cite[Theorem 3.11]{moonenmt} that for any polarizable $\mathbb{Q}$-Hodge structure $V$ of weight $w$, the length of the highest weight of $\rho_i$ is bounded above by the integer $N$, where $N+1$ is the number of integers $p$ such that $V^{p,w-p}\ne0$ in the Hodge decomposition of $V_\mathbb{C}$. A $\mathbb{Q}$-Hodge structure with Hodge numbers $(n,0,\ldots,0,n)$ thus has $N=1$. So the highest weight $\lambda_i$ of $\rho_i$ must have length $1$, meaning that $\lambda_i$ is minuscule and $\mathfrak{g}_i$ is of classical type \cite[Section 3.6]{moonenmt}. Table \ref{selfdual} then yields the rest of the result, making use, for Item \ref{binom}, of the combinatorial fact that, for any integer $z\ge 0$, the binomial coefficient ${{2z}\choose z}$ is congruent to 2 mod 4 if and only if $z$ is a power of 2.
\end{proof}
\begin{lem}\label{goursat}
Suppose that the Hodge group $H=Hg(V)$ is semisimple and that for each of irreducible representations $W_\sigma$ for $\sigma \in \Sigma(F)$, there is only one nontrivial $\rho_i$. If none of these representations $\rho_i$ are of type $D_4$, then we have the equality
\[\mathrm{Lie}(H)_\mathbb{C}=\prod_{\sigma\in \Sigma(F)}(\mathfrak{g}_i)_\sigma.\]
\end{lem}

\begin{proof}
For any polarizable $\mathbb{Q}$-Hodge structure $U$ and any positive integer $m\ge 1$, the Hodge group of the direct sum of $m$ copies of $U$ is isomorphic to the Hodge group of $U$ acting diagonally on $mU$  \cite[Remark 1.8]{moonenmt}. Thus for each of the representations $W_\sigma$, we have $Hg(qW_\sigma)=Hg(W_\sigma)$. Hence the decompositions of (\ref{DL0}) and (\ref{DL1}) yield the inclusion 
\begin{equation}\label{lieinc}\mathrm{Lie}(H)_\mathbb{C}\subseteq\prod_{\sigma\in \Sigma(F)} (\mathfrak{g}_i)_\sigma,\end{equation} where, in addition, we know that $\mathrm{Lie}(H)_\mathbb{C}$ surjects onto each of the $(\mathfrak{g}_i)_\sigma$ factors. 

If this inclusion is strict, then there must exist an isomorphism between two factors $(\mathfrak{g}_i)_{\sigma_1}$ and $(\mathfrak{g}_i)_{\sigma_2}$ whose graph gives the representation $(\rho_i)_{\sigma_1}\oplus (\rho_i)_{\sigma_2}$. However, since $(\mathfrak{g}_i)_{\sigma_1}$ is not of type $D_4$, its outer automorphism group is either trivial or is $\mathbb{Z}/2\mathbb{Z}$ \cite[Chapter 20]{fulton}. 
Hence, an isomorphism between $(\mathfrak{g}_i)_{\sigma_1}$ and $(\mathfrak{g}_i)_{\sigma_2}$ is induced by conjugation by an isomorphism between the underlying representations $W_{\sigma_1}$ and $W_{\sigma_2}$. But this contradicts the fact that the representations $W_\sigma$ were defined according to the decomposition  $L\otimes_\mathbb{Q}\mathbb{C} \cong \prod_{\sigma\in \Sigma(F)} M_q(\mathbb{C})$. Hence no such isomorphism exists and so the inclusion in (\ref{lieinc}) must be an equality. 
\end{proof}

\section{$SL(2)$-factors and the Hodge Group}
Let $V$ be a simple polarizable $\mathbb{Q}$-Hodge structure of weight $w\ge1$ with Hodge numbers $(n,0,\ldots, 0,n)$ and endomorphism algebra $L$.  Recall the definition of the Mumford-Tate group $M=MT(V)$ of $V$ as the smallest $\mathbb{Q}$-algebraic group contained in $GL(V)$ such that the homomorphism
\[\gamma:=h\circ \mu\colon \mathbb{G}_{m,\mathbb{C}}\rightarrow GL(V)_\mathbb{C}\]
factors through $M_\mathbb{C}$. Defining the representation $\phi$ to be the tautological representation
\[\phi\colon M\rightarrow GL(V),\]
the weights in $V_\mathbb{C}$  of the composition $\phi\circ \gamma$ are exactly the cocharacters $z\mapsto z^{-p}$, for $p$ an integer such that $V^{p,w-p}\ne 0$, meaning that the only two possible weights are  $z\mapsto z^{-w}$ or $z\mapsto 1$.

Observe that if the Hodge group $H=Hg(V)$ has decomposition $H_\mathbb{C}=Z(H)_\mathbb{C}\cdot H_1\cdots H_s,$
then the Mumford-Tate group has decomposition
\[M_\mathbb{C}=Z(M)_\mathbb{C}\cdot H_1\cdots H_s.\]

Passing to Lie algebras and writing $\mathfrak{c}=\mathrm{Lie}(Z(H))_\mathbb{C}$, $\mathfrak{c'}=\mathrm{Lie}(Z(M))_\mathbb{C}$, and $\mathfrak{g}_i=\mathrm{Lie}(H_i)$ yields
\[\mathrm{Lie}(H)_\mathbb{C}=\mathfrak{c}\times \mathfrak{g}_1\times \cdots \times \mathfrak{g}_s\]
\[\mathrm{Lie}(M)_\mathbb{C}=\mathfrak{c'}\times \mathfrak{g}_1\times \cdots \times \mathfrak{g}_s.\]

Now, for $\sigma \in \Sigma(F)$, let $W_\sigma$ be an irreducible representation of $M_\mathbb{C}$ defined as in Section \ref{section5} with decomposition as a representation of $\mathrm{Lie}(M)_\mathbb{C}$ given by
\begin{equation}\label{irrepW}W=\chi\boxtimes \rho_1\boxtimes\cdots \boxtimes \rho_s,\end{equation}
where $\chi$ is a character of $\mathfrak{c}'$ and the $\rho_i$ are irreducible representations of the simple factors $\mathfrak{g}_i$. 

Write the weights in $W_\sigma$ of the homomorphism $\gamma\colon \mathbb{G}_{m,\mathbb{C}}\rightarrow M_\mathbb{C}$ as $z\mapsto (z^{-l_0}, z^{-l_1}, \ldots, z^{-l_s})$ according to the decomposition $W_\sigma=\chi\boxtimes \rho_1\boxtimes\cdots \boxtimes \rho_s$.

\begin{lem}\label{sl2lem}
Suppose that $\mathfrak{g}_i\cong \mathfrak{sl}_2$ for some $1\le i \le s$. Denote the two possible values of $l_i$ by $\alpha$ and $\beta$ and denote the possible values of $\left(\sum_{j=0}^s l_j\right)-l_i$ by $\lambda_1,\ldots, \lambda_r$ for $r=\frac{\dim_\mathbb{C}W_\sigma}{2}$. Then either $\alpha=\beta=0$ or $\lambda_1=\cdots=\lambda_r=0$.
\end{lem} 

\begin{proof}
The only two weights of $\phi\circ \gamma$ in $W_\sigma$ are the cocharacters $z\mapsto z^{-w}$ and $z\mapsto 1$. Hence for any weight of $\gamma\colon \mathbb{G}_{m,\mathbb{C}}\rightarrow M_\mathbb{C}$ written as $z\mapsto (z^{-c}, z^{-l_1}, \ldots, z^{-l_s})$, we must have the sum
$\sum_{j=1}^s l_j$ either equal to $w-c$ or equal to $0$. In other words, half of the elements in the set
\[\{\alpha+\lambda_k\mid 1\le k\le r\} \cup \{\beta+\lambda_k\mid 1\le k\le r\}\]
are equal to $0$ and half are equal to $w-c$.

If $\alpha+\lambda_k=0$ and $\beta+\lambda_k=w-c$ for all $1\le k \le r$, then $\lambda_k=-\alpha=w-c-\beta$ for all $1\le k\le r$. But then the representation $\boxtimes_{k\ne i}\rho_k$ is just multiplication by $z^{\alpha}$. But $\prod_{k\ne i}\mathfrak{g}_k$ is contained in $\mathfrak{sl}_{r}$ and hence $\sum_{k=1}^r\lambda_k=r(-\alpha)=0$. Thus $\lambda_k=0$ for all $1\le k\le r$ and we are done. 

If it is not the case that $\alpha+\lambda_k=0$ and $\beta+\lambda_k=w-c$ for all $1\le k \le r$, then there exists $t\in \{1,\ldots, r\}$ such that without loss of generality $\lambda_1=\cdots=\lambda_t=-\alpha$ and $\lambda_{t+1}=\cdots=\lambda_r=w-c-\alpha$. Adding $\beta$ yields $\beta-\alpha=0$. Then since $\rho_i$ acts as multiplication by $z^{-\alpha}$ and $\mathfrak{g_i}=\mathfrak{sl}_2$, we have $2\alpha$=0, hence $\alpha=\beta=0$. 
\end{proof}

\begin{cor}\label{sl2cor}
There is no simple polarizable $\mathbb{Q}$-Hodge structure $V$ of weight $w\ge1$ with Hodge numbers $(n,0,\ldots, 0,n)$ such that
\[Hg(V)_\mathbb{C}=SL(2,\mathbb{C})\times G,\]
where $G$ is a nontrivial semisimple $\mathbb{C}$-group having no simple factors isomorphic to $SL(2,\mathbb{C})$ and $Hg(V)_\mathbb{C}$ acts irreducibly on $V_\mathbb{C}$ by the product of the standard representation of $SL(2,\mathbb{C})$ with a representation of $G$.
\end{cor}
\begin{proof}
If such a Hodge structure $V$ existed, then by Lemma \ref{sl2lem}, the homomorphism $\gamma$ would actually factor through either $Z(M)_\mathbb{C}\cdot SL(2,\mathbb{C})$ or through $Z(M)_\mathbb{C}\cdot G$. But then, since $G$ has no simple factors isomorphic to $SL(2,\mathbb{C})$, the Hodge group would be a $\mathbb{Q}$-form of $SL(2)$ or a $\mathbb{Q}$-form of $G$, which is a contradiction.
\end{proof}

\section{Lower Bound on the Rank of the Hodge Group}

The following lemma restates a result proved by Orr \cite[Theorem 1.1]{orr}
for abelian varieties, which generalized an earlier result by Ribet \cite{ribet2} for abelian varieties of CM-type.

\begin{lem}\label{bound}
Let $V$ be a polarizable $\mathbb{Q}$-Hodge structure of weight $w\ge1$ with Hodge numbers $(n,0,\ldots, 0,n)$ whose endomorphism algebra $L$ is commutative. Then the rank, as a $\mathbb{Q}$-algebraic group, of the Hodge group $Hg(V)$ satisfies
$$\mathrm{Rank}(Hg(V))\ge\mathrm{log}_2(2n).$$
\end{lem}

\begin{proof}
As in previous sections, for the Mumford-Tate group $M=MT(V)$,  we have decompositions
$\mathrm{Lie}(M)_\mathbb{C}=\mathfrak{c}\times \mathfrak{g}_1\times \cdots \times \mathfrak{g}_s$ and $W_\sigma=\chi\boxtimes \rho_1\boxtimes\cdots \boxtimes \rho_s$, where $W_\sigma$ is an irreducible  $M_\mathbb{C}$-module coming from the decomposition of $L\otimes_\mathbb{Q}\mathbb{C}$ according to the embeddings $\sigma\in \Sigma(F)$. By Lemma \ref{tablelem}, the nontrivial $\rho_i$ have highest weights which are minuscule and so their weight spaces are all one-dimensional.

Let $T$ be a maximal torus of $M$ and let $r$ be its rank. Consider the restricted representation on $W_\sigma$ given by $(\chi\boxtimes \rho_1\boxtimes\cdots \boxtimes \rho_s)\mid_T$. Since the characters of $T$ in a minuscule representation have multiplicity $1$, we know
\[\dim W_\sigma=(\text{ number of characters of } (\chi\boxtimes \rho_1\boxtimes\cdots \boxtimes \rho_s)\mid_T).\]

Moreover, since $L$ is commutative, none of the representations $W_\sigma$ indexed by the embeddings $\sigma\in \Sigma(F)$ can be isomorphic to each other and so, since non-isomorphic minuscule representations have disjoint characters, these $W_\sigma$ have disjoint characters. Letting $\phi\colon M\rightarrow GL(V)$ be the tautological representation, the sum $V_\mathbb{C}$ of the representations $W_\sigma$ satisfies $\dim V_\mathbb{C}=2n$, so we have
\begin{equation}\label{boundineq1}2n=(\text{ number of characters of }\phi|_T).\end{equation}

Let $S$ be the set of $\mathrm{Aut}(\mathbb{C}/\mathbb{Q)}$-conjugates of the homomorphism $\gamma$ and let $S'$ be the set of cocharacters of the torus $T_\mathbb{C}$ that are $M(\mathbb{C})$-conjugate to an element of $S$. Note that the images of the $M(\mathbb{C})$-conjugates of elements of $S$ generate $M_\mathbb{C}$.  Moreover, because every cocharacter of $M$ is $M(\mathbb{C})$-conjugate to a cocharacter of $T_\mathbb{C}$, the images of the $M(\mathbb{C})$-conjugates of elements of $S'$ still generate $M_\mathbb{C}$. 

Consider the action of the Weyl group of $M$ on 
$X_*(T)\otimes_\mathbb{Z} \mathbb{Q}=\mathrm{Hom}(\mathbb{G}_m, T_\mathbb{C})\otimes_ \mathbb{Z}\mathbb{Q}.$
Since $S'$ is closed under this action  and $M_\mathbb{C}$ is generated by the images of the $M(\mathbb{C})$-conjugates of elements of $S'$, it must be the case that $S'$ spans $X_*(T)\otimes_\mathbb{Z} \mathbb{Q}$ as a $\mathbb{Q}$-vector space. Thus, let $\Delta$ be a basis for $X_*(T)\otimes_\mathbb{Z} \mathbb{Q}$ contained in $S'$.  So, $|\Delta|$ is equal to the rank $r$ of $T$. 

Any weight $\lambda \in \mathrm{Hom}(T_\mathbb{C}, \mathbb{G}_m)$ of the representation $\phi$ is then determined by the integers $\langle \lambda, \delta \rangle$ for $\delta \in \Delta$. But any $\delta$ in $\Delta$ is also in $S'$ and thus is some $\mathrm{Aut}(\mathbb{C}/\mathbb{Q)}$- and $M(\mathbb{C})$-conjugate of $\gamma$. 

Recall that for any $z\in \mathbb{C}^*$, the map $\phi\circ \gamma$ acts as multiplication by $z^{-w}$ on $V^{w,0}$ and as the identity on $V^{0,w}$. Hence, if $\delta \in \Delta$, the integer $\langle \lambda, \delta \rangle$ can only be 0 or $-w$. Namely,
\[(\text{ number of characters of }\phi|_T)\le 2^r.\]
However, this bound may be reduced by noting that $M$ contains the homotheties. Namely, there is a unique cocharacter 
$\nu:\mathbb{G}_m\rightarrow M_\mathbb{C}$ such that for $z\in \mathbb{C}^*$, the map $\phi\circ \nu$ acts by multiplication by $z^{-w}$. So, in fact, $\nu$ may be viewed as an element of $X_*(T)\otimes_\mathbb{Z} \mathbb{Q}$ and, moreover, $\langle \lambda, \nu \rangle$ is equal to $-w$ for any character $\lambda$ of $\phi\mid_T$. We may thus choose a new subset $\Delta'$ of $X_*(T)\otimes_\mathbb{Z} \mathbb{Q}$ and $S'$ such that $\nu \cup \Delta'$ forms a basis of $X_*(T)\otimes_\mathbb{Z} \mathbb{Q}$.  Repeating the above arguments for $\Delta'$ then yields
\begin{equation}\label{boundineq2}(\text{number of characters of }\phi|_T)\le 2^{r-1}.\end{equation}
Combining (\ref{boundineq1}) and (\ref{boundineq2}) yields
\[\mathrm{log}_2(2n)\le r-1,\]
where $r$ is the rank of $M$ as an algebraic group over $\mathbb{Q}$. However we know $M$ is the almost direct product inside $GL(V)$ of $\mathbb{G}_{m}$ and $Hg(V)$, so $Hg(V)$ has rank $r-1$.
\end{proof}

\section{Hodge Representations and Mumford-Tate domains}\label{hgreps}

Following \cite[Section IV.A]{domain}, we introduce the notions of Hodge representations and Mumford-Tate domains, which will prove useful in later sections. 

\begin{defn}Suppose $V$ is a $\mathbb{Q}$-vector space and $\langle,\rangle\colon V\otimes V\rightarrow \mathbb{Q}$ is a bilinear form on $V$ such that $\langle u,v\rangle =(-1)^w\langle v, u\rangle$ for some integer $w\ge 1$. A \emph{Hodge representation} $(H,\rho,\phi)$ is the data of a  representation defined over $\mathbb{Q}$
\[\rho\colon H\rightarrow \mathrm{Aut}(V, \langle, \rangle)\]
of a connected $\mathbb{Q}$-algebraic group $H$ and a non-constant homomorphism
\[\phi\colon U_1\rightarrow H_\mathbb{R}\]
such that the data $(V, \langle,\rangle, h:=\rho\circ \phi)$ is a polarized $\mathbb{Q}$-Hodge structure of weight $w$. 
\end{defn}

Let $(H,\rho,\phi)$ be a Hodge representation attached to the $\mathbb{Q}$-vector space $V$ with bilinear form $\langle, \rangle$ and let $h=\rho \circ \phi$ be its associated polarized $\mathbb{Q}$-Hodge structure of weight $w$. Let $D$ be the period domain of all $\mathbb{Q}$-Hodge structures with the same Hodge numbers as $V$. Consider the Mumford-Tate domain $D_{H_h}$ given by the $H(\mathbb{R})$-orbit of the Hodge structure $h$ inside of the period domain $D$. Let $D_{H_h}^0$ be a connected component of this Mumford-Tate domain $D_{H_h}$.

Note that $D_{H_h}^0$ is a closed analytic space and thus we may speak of a \emph{very general} $\mathbb{Q}$-Hodge structure in $D_{H_h}^0$ to mean a $\mathbb{Q}$-Hodge structure in $D_{H_h}^0$ occurring outside of countably many closed analytic subspaces not equal to $D_{H_h}^0$.


\begin{lem}\label{mtdomainlem}A very general $\mathbb{Q}$-Hodge structure in $D_{H_h}^0$ has Hodge group a connected normal $\mathbb{Q}$-subgroup of $H$. 
\end{lem}
\begin{proof}
The following proof is heavily inspired by the very similar proofs of Totaro's \cite[Page 4110]{totaro} and Green-Griffiths-Kerr \cite[Proposition VI.A.5]{domain}. From \cite[Proposition IV.A.2]{domain}, every $\mathbb{Q}$-Hodge structure in $D_{H_h}^0$ has Hodge group contained in the group $H$. Moreover, the Hodge group $G$ of a very general $\mathbb{Q}$-Hodge structure in $D_{H_h}^0$ is determined by the data $(V,\langle,\rangle, \rho, D_{H_h}^0)$. Using that $D_{H_h}^0$ is connected, the action of the group $H(\mathbb{Q})$ preserves this data and hence normalizes the algebraic group $G$. Since $H$ is a connected group over the perfect field $\mathbb{Q}$, the group $H(\mathbb{Q})$ is Zariski dense in $H$ \cite[Corollary 18.3]{borel}. Thus, in fact, the group $H$ normalizes the group $G$. Since $G\subset H$, we have that $G$ is a connected normal $\mathbb{Q}$-subgroup of $H$. 
\end{proof}


\begin{cor}\label{mtdomaincor} Let $(H,\rho,\phi)$ be a Hodge representation of a polarized $\mathbb{Q}$-Hodge structure $(V, \langle,\rangle, h:=\rho\circ \phi)$ such that there exists $p\ne q$ such that $V^{p,q}\ne 0$  in the Hodge decomposition of $V_\mathbb{C}$. If $H$ is a $\mathbb{Q}$-simple group, then a very general $\mathbb{Q}$-Hodge structure in $D_{H_h}^0$ has Hodge group equal to $H$.
\end{cor}
\begin{proof}
By Lemma \ref{mtdomainlem}, since $H$ is $\mathbb{Q}$-simple, the Hodge group $G$ of a very general $\mathbb{Q}$-Hodge structure in $D_{H_h}^0$ is either $1$ or all of $H$. But since there exists $p\ne q$ such that $V^{p,q}$ is nonzero, by Remark \ref{triv}, the group $G$ must be nontrivial. 
\end{proof}

\section{Hodge Groups for Type I Endomorphism Algebras}\label{TypeIsection}
We now begin our characterization of Hodge groups of simple polarizable $\mathbb{Q}$-Hodge structures with Hodge numbers $(n,0,\ldots,0,n)$, making use of the notation for Hodge groups and Lefschetz groups established in Section \ref{lefschetz}.

\begin{prop}\label{rational}
Let $V$ be a simple polarizable $\mathbb{Q}$-Hodge structure of weight $w\ge1$ with Hodge numbers $(n,0,\ldots, 0,n)$ and endomorphism algebra $L$ a totally real number field such that $l=\frac{n}{[L:\mathbb{Q}]}$ is odd. Then, 
\begin{equation*}
Hg(V)=
\begin{cases}
R_{L/\mathbb{Q}}Sp(_LV) &\mbox{if }w \text{ is odd}\\
R_{L/\mathbb{Q}}SO(_LV) \text{ or } R_{L/\mathbb{Q}}SU(2^k) \text{ } \left(\text{for }k\ge3 \text{ and } 2l={{2^k}\choose {2^{k-1}}}\right)&\mbox{if }w \text{ is even.}
\end{cases}
\end{equation*} 
Moreover, if $w$ is even, then $l\ge 3$ and both possible groups occur. \end{prop}

\begin{proof}
When $V$ is of odd weight, the result follows using the equivalence of Remark \ref{equiv} together with a result of Ribet's \cite[Theorem 1]{ribet1}, which proves the result for simple complex abelian varieties. Thus we may assume that the weight $w$ of $V$ is even.

The case when $w$ is even and $l=1$ is exceptional case (6) in Totaro's classification of the of the possible endomorphism algebras of $\mathbb{Q}$-Hodge structures of the specified type (see Section \ref{endclass}), and thus, since $l$ is assumed to be odd, we know $l\ge 3$.  

Let $H=Hg(V)$ be the Hodge group of $V$, which by Remark \ref{ss} is semisimple. Remark \ref{lefrem} and Table \ref{exphg} about the Lefschetz group of $V$ imply
\[H\subseteq R_{L/\mathbb{Q}}SO(_LV).\]

Now, in the notation of Section \ref{section5}, consider an irreducible representation $W_\sigma=\rho_1\boxtimes\cdots \boxtimes \rho_s$ of $\mathrm{Lie}(H)_\mathbb{C}$ induced by the decomposition $L\otimes_\mathbb{Q}\mathbb{C}=\prod_{\sigma\in \Sigma(L)}\mathbb{C}$.

Note that since the representation $W_\sigma$ is orthogonal, each of the nontrivial $\rho_i$ is a self-dual representation, meaning either symplectic or orthogonal, and the number of $i$ such that $\rho_i$ is symplectic must be even. Since $W_\sigma$ has dimension $2l$ with $l$ odd, using Part (\ref{TL2}) of Lemma \ref{tablelem}, none of the dimensions of the representations $\rho_i$ can be divisible by $4$ and there can be only one $i$ such that $\rho_i$ is nontrivial. Hence this nontrivial $\rho_i$ must be orthogonal. Applying Part (\ref{TL4}) of Lemma \ref{tablelem} yields that this nontrivial $\rho_i$ is either of type $D_l$, acting on $W_\sigma$ by the standard representation of $\mathfrak{so}_{2l}$ or, in the case that $2l={{2^k}\choose {2^{k-1}}}$ for some $k\ge3$, of type $A_{2^k-1}$, acting on $W_\sigma$ by the $(2^k-1)$-th exterior product of the standard representation of $\mathfrak{sl}_{2k}$.

Applying Lemma \ref{goursat} then yields
\[\mathrm{Lie}(H)_\mathbb{C}=\prod_{\sigma\in \Sigma(L)}\mathfrak{g}_\sigma,\]
where $\mathfrak{g}_\sigma$ is either equal to $\mathfrak{so}_{2l}$ or to $\mathfrak{sl}_{2^k}$ for $2l={{2^k}\choose {2^{k-1}}}$ for some $k\ge3$. But since the group $H$ must be defined over $\mathbb{Q}$, we must have either $\mathfrak{g}_\sigma=\mathfrak{so}_{2l}$ for all $\sigma\in \Sigma(L)$ or $\mathfrak{g}_\sigma=\mathfrak{sl}_{2^k}$ for all $\sigma\in \Sigma(L)$.

Hence either $H=R_{L/\mathbb{Q}}SO(_LV)$, which is the generic case \cite[Corollary II.A.6]{domain}, or $H=R_{L/\mathbb{Q}}SU(2^k)$ acting via the representation $\rho\colon R_{L/\mathbb{Q}}SU(2^k)\rightarrow SO(V)$ given by the product over all embeddings $\sigma\in \Sigma(L)$ of the $2^k-1$-th exterior product of the standard representation of $SU(2^k)$. We must show that this second case really can occur. 

Let $r=[L:\mathbb{Q}]$ and consider the homomorphism
\[\phi\colon \mathbb{U}_1\rightarrow SU(2^k,\mathbb{R})^r\]
given by
\[z\in \mathbb{C}^*\mapsto \left(
\begin{array}{c|c}
z^{-\frac{w}{2^{k-1}}}\cdot \mathrm{Id}_{2^k-1} & 0\\
\hline
0 & z^{\frac{(2^k-1)w}{2^{k-1}}}
\end{array}
\right)^r.\]
Let $h$ be the composition $h=\rho\circ \phi$. Then observe that $h$ defines a weight $w$ polarized Hodge structure with Hodge numbers $(n,0,\ldots,0,n)$ on $V$. Namely, the data $(R_{L/\mathbb{Q}}SU(2^k),\rho,\phi)$ defines a Hodge representation with Hodge numbers $(n,0,\ldots,0,n)$. 

Since $R_{L/\mathbb{Q}}SU(2^k)$ is $\mathbb{Q}$-simple, by Corollary \ref{mtdomaincor} a very general $\mathbb{Q}$-Hodge structure in a connected component of the Mumford-Tate domain $D_{R_{L/\mathbb{Q}}SU(2^k)_h}$ will have Hodge group equal to $R_{L/\mathbb{Q}}SU(2^k)$.

\end{proof}


\begin{prop}\label{2TypeI}
Let $V$ be a simple polarizable $\mathbb{Q}$-Hodge structure of weight $w\ge1$ with Hodge numbers $(n,0,\ldots, 0,n)$ and endomorphism algebra $L$ a totally real number field such that $\frac{n}{[L:\mathbb{Q}]}=2$. Then

\begin{equation*}
Hg(V)=
\begin{cases}
R_{L/\mathbb{Q}}Sp(_LV) &\mbox{if }w \text{ is odd}\\
R_{L/\mathbb{Q}}SO(_LV)&\mbox{if }w \text{ is even.}
\end{cases}
\end{equation*}
\end{prop}

\begin{proof}

Let $H=Hg(V)$ be the Hodge group of $V$, which is semisimple by Remark \ref{ss}. So then, using Remark \ref{lefrem} and Table \ref{exphg}, when $w$ is even (respectively when $w$ is odd), we have
\[\begin{array}{cc}
H\subseteq R_{L/\mathbb{Q}}SO(_LV) & \text{ (respectively  }H\subseteq R_{L/\mathbb{Q}}Sp(_LV) \text{)}.
\end{array}\]

As in the proof Proposition \ref{rational} and using the notation of Section \ref{section5}, let $W_\sigma=\rho_1\boxtimes\cdots \boxtimes \rho_s$ be an irreducible representation of $\mathrm{Lie}(H)_\mathbb{C}$ induced by the decomposition $L\otimes_\mathbb{Q}\mathbb{C}=\prod_{\sigma\in \Sigma(L)}\mathbb{C}$. These $W_\sigma$ are $4$-dimensional and are orthogonal (respectively symplectic), which implies that each nontrivial $\rho_i$ is self-dual and the number of $i$ such that $\rho_i$ is symplectic must be even (respectively odd). Hence Lemma \ref{tablelem} yields that the representation $W_\sigma$ is of type $A_1\times A_1$ acting  by the product of the standard representations of $\mathfrak{sl}_2$ (respectively of type $C_2$ acting by the standard representation of $\mathfrak{sp}_{4}$).

In the latter case, namely when $w$ is odd and hence the representation $W_\sigma$ is of type $C_2$, Lemma \ref{goursat} yields $\mathrm{Lie}(H)_\mathbb{C}=\prod_{\sigma\in \Sigma(L)}\mathfrak{sp}_4$ and so $H=R_{L/\mathbb{Q}}Sp(_LV).$

Now consider the case when $w$ is even and hence the representation $W_\sigma$ is of type $A_1\times A_1$.  If the Hodge structure $V$ is such that the Lefschetz group $R_{L/\mathbb{Q}}SO(_LV)$ is simple, then either $H$ is all of $R_{L/\mathbb{Q}}SO(_LV)$ or, as in the proof of Lemma \ref{goursat}, there is a $\mathrm{Lie}(H)_\mathbb{C}$-module isomorphism $\alpha$ between factors $\mathfrak{so}_4$ and $\mathfrak{so}_4$. Such an isomorphism $\alpha$ may be viewed as a matrix 
\begin{equation}\label{goursatmatrix}
\left(
\begin{array}{cc}
\psi_{11} & \psi_{12}\\

\psi_{21} & \psi_{22}
\end{array}
\right),
\end{equation} 
where each $\psi_{ij}\in \mathrm{Aut}(\mathfrak{sl}_2)$.  Since automorphisms of $\mathfrak{sl}_2$ are all inner, as in the proof of Lemma \ref{goursat}, the nonzero $\psi_{ij}$ induce isomorphisms of the standard representations of their corresponding $\mathfrak{sl}_2$ factors. Namely, if $\alpha$ is an isomorphism between the copy of $\mathfrak{so}_4$ acting on $W_1=U_{11}\otimes U_{12}$ and the copy of $\mathfrak{so}_4$ acting on $W_2=U_{21}\otimes U_{22}$, where each $U_{ij}$ is the standard representation of $\mathfrak{sl}_2$, then each nonzero $\psi_{ij}$ induces a $\mathrm{Lie}(H)_\mathbb{C}$-module isomorphism between $U_{1j}$ and $U_{2i}$. Since the matrix in (\ref{goursatmatrix}) is invertible, we get $U_{11}\cong U_{21}, U_{12}\cong U_{22}$ or $U_{11}\cong U_{22}, U_{12}\cong U_{21}$, which in either case yields $W_1\cong W_2$, contradicting the assumption that the endomorphism algebra $L$ is a field. Hence, if $R_{L/\mathbb{Q}}SO(_LV)$ is simple, then the Hodge group $H$ is all of $R_{L/\mathbb{Q}}SO(_LV)$.

Now suppose that the Lefschetz group $R_{L/\mathbb{Q}}SO(_LV)$ is not simple. This occurs when $_LV$ has discriminant $1$ in $L^*/(L^*)^2$ and thus $SO(_LV)$ is the product of two subgroups $SL(1,L')$ and $SL(1,L'^{\mathrm{op}})$, where $L'$ is a quaternion algebra over $L$ \cite[Corollary 15.12]{boi}. So then we know that the Hodge group $H$ as a $\mathbb{Q}$-group satisfies
\begin{equation} \label{nonsimpledecomp}H\subset SL(1,L')\cdot SL(1,L'^{\mathrm{op}})\end{equation}
and that $H$ surjects onto each of these two factors. Hence $H$ is either the entire product or $H$ is the graph of an isomorphism between the two simple factors.  

The Mumford-Tate domain $D$ of Hodge structures with Hodge numbers $(n,0\ldots,0,n)$ and endomorphism algebra contained in $L$ is isomorphic to $(\mathbb{C}\mathbb{P}^1\sqcup \mathbb{C}\mathbb{P}^1)^g$, where $g=[L:\mathbb{Q}]$. Consider the homomorphism $\gamma\colon \mathbb{G}_{m,\mathbb{C}}\rightarrow MT(V)_\mathbb{C}$ defining the Mumford-Tate group of $V$. Then by Lemma \ref{sl2lem}, the weights in each irreducible representation $W_{\sigma}=U_{\sigma,1}\otimes U_{\sigma,2}$ of $\gamma$ must either be trivial on $U_{\sigma,1}$ or on $U_{\sigma,2}$.  

Observe that if the Hodge group $H$ is the graph of an isomorphism between the simple factors $SL(1,L')$ and $SL(1,L'^{\mathrm{op}})$ in (\ref{nonsimpledecomp}), then either the weights of $\gamma$ on $U_{\sigma,1}$ are trivial for all $\sigma \in \Sigma(L)$ or the weights of $\gamma$ on $U_{\sigma,2}$ are trivial for all $\sigma \in \Sigma(L)$.  Namely the Hodge structure $V$ lies exactly on the two connected components of the Mumford-Tate domain $D$ whose generic elements are non-simple, as proved by Totaro in \cite[Theorem 4.1]{totaro}. Namely, if $H$ is the graph of an isomorphism between the two factors in (\ref{nonsimpledecomp}), then the Hodge structure $V$ is not simple, which is a contradiction. Hence the Hodge group $H$ is all of $SL(1,L')\cdot SL(1,L'^{\mathrm{op}})\cong R_{L/\mathbb{Q}}SO(_LV)$.
\end{proof}


\begin{prop}\label{twiceodd}
Let $V$ be a simple polarizable $\mathbb{Q}$-Hodge structure of weight $w\ge1$ with Hodge numbers $(n,0,\ldots, 0,n)$, where $n$ is twice an odd number, such that the endomorphism algebra $L$ of $V$ is equal to $\mathbb{Q}$. Then
\begin{equation*}
Hg(V)=
\begin{cases}
Sp(V) &\mbox{if }w \text{ is odd}\\
SO(V) \text{ or } R_{L/\mathbb{Q}}SU(2^k) \text{ } \left(\text{for }k\ge3 \text{ and } 2n={{2^k}\choose {2^{k-1}}}\right)&\mbox{if }w \text{ is even.}
\end{cases}
\end{equation*} 
In the case when $w$ is even, both possible groups occur.
\end{prop}

\begin{proof}
Let $H=Hg(V)$ be the Hodge group of $V$. Then, using Remark \ref{lefrem} and Table \ref{exphg}, when $w$ is even (respectively when $w$ is odd), we have
\[\begin{array}{cc}
H\subseteq SO(V) & \text{ (respectively  }H\subseteq Sp(V) \text{)}.
\end{array}\]
Since the endomorphism algebra $L$ is equal to $\mathbb{Q}$, the representation $V$ of $H$ is irreducible. Applying Lemma \ref{tablelem}, the possibilities for $H$ acting on $V$ are:
\begin{enumerate}
\item $SO(2n)$ acting by the standard representation
\item $SO(2^k)$, with $2n={{2^k}\choose {2^{k-1}}}$ for $k\ge3$, acting by $\bigwedge^{2^{k-1}}(\mathrm{Standard})$
\item $SU(2)\times SO(n)$ acting by the tensor product of the two standard representations
\end{enumerate}
(respectively,
\begin{enumerate}
\item $Sp(2n)$ acting by the standard representation
\item $SL(2)\times SO(n)$ acting by the tensor product of the two standard representations
\item $SL(2)\times SL(2^k)$, with $n={{2^k}\choose {2^{k-1}}}$ for $k\ge3$, acting by the standard representation of $SL(2)$ tensor $\bigwedge^{2^{k-1}}(\mathrm{Standard}).$
\end{enumerate}

However, using Corollary \ref{sl2cor} we may eliminate $SL(2)\times SO(n)$ (respectively $SL(2)\times SO(n)$ and $SL(2)\times SL(2^k)$) as possibilities. Hence $H_\mathbb{C}$ must be $SO(V)\cong SO(2n)$ or $SO(2^k)$ acting by $\bigwedge^{2^{k-1}}(\mathrm{Standard})$
  (respectively $Sp(V)\cong Sp(2n)$). 

As in the proof of Proposition \ref{rational}, the case $H=SO(V)$ is the generic case and thus is always possible \cite[Corollary II.A.6]{domain}. Thus, it remains to show that when $w$ is even and $2n={{2^k}\choose {2^{k-1}}}$ for $k\ge3$, the Hodge group $SU(2^k)$ acting by the representation 
\[\rho\colon R_{L/\mathbb{Q}}SU(2^k,L)\rightarrow SO(V),\]
given by the $2^k-1$-th exterior product of the standard representation of $SU(2^k,L)$, is also possible. However, as in the proof of Proposition \ref{rational}, the homomorphism
\[\phi\colon \mathbb{U}_1\rightarrow SU(2^k,\mathbb{R})\]
given by
\[z\in \mathbb{C}^*\mapsto \left(
\begin{array}{c|c}
z^{-\frac{w}{2^{k-1}}}\cdot \mathrm{Id}_{2^k-1} & 0\\
\hline
0 & z^{\frac{(2^k-1)w}{2^{k-1}}}
\end{array}
\right)\]
defines a Hodge representation $(SU(2^k),\rho,\phi)$ with Hodge numbers $(n,0,\ldots,0,n)$. Since the $\mathbb{Q}$-group $SU(2^k)$ is simple, Corollary \ref{mtdomaincor}  finishes the proof. 
\end{proof}


\begin{prop}\label{four1} Let $V$ be a simple polarizable $\mathbb{Q}$-Hodge structure of weight $w\ge1$ with Hodge numbers $(4,0,\ldots,0,4)$ and endomorphism algebra $L$ equal to $\mathbb{Q}$. Then
\begin{equation*}
Hg(V)=
\begin{cases}
 Sp(V) \text{ or } SL(2)\times SO(4) \text{ acting by product of standard representations} &\mbox{if } w \text{ is odd}\\
 SO(V) \text{ or } SO(7) \text{ acting by spin representation }
 &\mbox{if } w \text{ is even.}
 \end{cases}
 \end{equation*}
Moreover, all of the above groups occur.
\end{prop}

\begin{proof}

When $V$ is of odd weight, the result follows using the equivalence of Remark \ref{equiv} together with the analogous statement for abelian fourfolds proved by Moonen and Zarhin \cite[4.1]{fourfold}. Thus we may assume that $V$ is of even weight.

In this case, using Table \ref{exphg}, the Lefschetz group of $V$ is $SO(V)\cong SO(8)$. Since the endomorphism algebra $L$ is equal to $\mathbb{Q}$, the representation $V$ of $H$ is irreducible. Applying Lemma \ref{tablelem}, the possibilities for $H$ acting on $V$ are:
\begin{enumerate}
\item $SO(8)$ acting by the standard representation, 
\item $SO(7)$ acting by the spin representation
\item $SL(2)\times Sp(4)$ acting by the tensor product of the two standard representations.
\end{enumerate}
However, using Corollary \ref{sl2cor} we may eliminate $SL(2)\times Sp(4)$ as a possibility. As in the preceding proofs, the case $H=SO(V)\cong SO(8)$ is the generic case and thus is always possible \cite[Corollary II.A.6]{domain}. Thus, it remains to show that the Hodge group $SO(7)$ acting by the spin representation is also possible. 

To construct the spin representation $\rho$, choose a pair 
$(W,W^*)$
of maximal isotropic subspaces of $\mathbb{Q}^7$ equipped with a symmetric bilinear form $(,)$ such that $W\cap W^*=0.$
If $a_1,a_2,a_3$ is a basis for $W$, then there is a unique basis $\alpha_1,\alpha_2,\alpha_3$ of $W^*$ such that for all $i$ and $j$
$(a_i,\alpha_j)=\delta_{ij}.$
Now consider the element $\psi_{x\wedge y}$ of the Lie algebra $\mathfrak{so}_7$ given by
\[\psi_{x\wedge y}(v)=2((y,v)x-(x,v)y)).\]
We can identify $\bigwedge^2\mathbb{Q}^7$ with  the Lie algebra $\mathfrak{so}_7$ via the map
\[x\wedge y\mapsto \psi_{x\wedge y}.\] 

Now let $\phi\colon U_1\rightarrow SO(7,\mathbb{R})$ be induced by the map on Lie algebras sending
\[1\mapsto 2w(\alpha_1\wedge a_1).\]
Then the data $(SO(7), \rho,\phi)$ defines a Hodge representation with Hodge numbers $(4,0,\ldots,0,4)$. Since $SO(7)$ is $\mathbb{Q}$-simple, Corollary \ref{mtdomaincor} finishes the proof.

\end{proof}

\section{Hodge Groups for Type II/III Endomorphism Algebras}\label{typeii/iii}

\begin{prop}\label{oddquat}
Let $V$ be a simple polarizable $\mathbb{Q}$-Hodge structure of weight $w\ge1$ with Hodge numbers $(n,0,\ldots,0,n)$ and endomorphism algebra $L$ of Type II or III such that $m=\frac{2n}{[L:\mathbb{Q}]}$ is odd. Let $B\cong M_m(L^{\mathrm{op}})$ be the centralizer of $L$ in $\mathrm{End}_\mathbb{Q}(V)$ and let $F$ be the center of $L$. Then
\begin{enumerate}
\item If $L$ is of Type II, then:
\begin{equation*}
Hg(V)=
\begin{cases}
R_{F/\mathbb{Q}}Sp(B,^{-}) &\mbox{if } w \text{ is odd}\\
R_{F/\mathbb{Q}}O^+(B,^{-}) \text{ or } R_{F/\mathbb{Q}}SU(2^k) \text{ } \left(\text{with } 2m={{2^k}\choose {2^{k-1}}} \text{ for } k\ge3\right) &\mbox{if } w \text{ is even.}
\end{cases}
\end{equation*}
\item If $L$ is of Type III, then:
\begin{equation*}
Hg(V)=
\begin{cases}
R_{F/\mathbb{Q}}O^+(B,^{-}) \text{ or } R_{F/\mathbb{Q}}SU(2^k) \text{ } \left( \text{with } 2m={{2^k}\choose {2^{k-1}}} \text{ for }k\ge3\right) &\mbox{if } w \text{ is odd}\\
R_{F/\mathbb{Q}}Sp(B,^{-}) &\mbox{if } w \text{ is even.}
\end{cases}
\end{equation*}
\end{enumerate}

Additionally, when $L$ is of Type II and $w$ is even or when $L$ is of Type III and $w$ is odd, then $m\ge3$ and both possible groups occur. 
\end{prop}

\begin{proof}
The restriction on $m$ when $L$ is of Type II and $w$ is even or when $L$ is of Type III and $w$ is odd follows from Totaro's classification (see Section \ref{endclass}).

Let $H=Hg(V)$ be the Hodge group of $V$, which by Remark \ref{ss} is semisimple. Remark \ref{lefrem} and Table \ref{exphg}, when $L$ is of Type II and $w$ is even or when L is of Type III and $w$ is odd (respectively when $L$ is of Type II and $w$ is odd or when L is of Type III and $w$ is even) yield
\[\begin{array}{cc}
H\subseteq R_{F/\mathbb{Q}}O^+(B,^{-}) & \text{ (respectively  }R_{F/\mathbb{Q}}Sp(B,^{-}) . \text{)}
\end{array}\]
Here $O^+(B,^{-})$ is an $F$-form of $SO(2m)$ and $Sp(B,^{-})$ is an $F$-form of $Sp(2m)$.

Thus, in the notation of Section \ref{section5}, the decomposition $L \otimes _\mathbb{Q} \mathbb{C}\cong \prod M_2(\mathbb{C})$ indexed by the set of embeddings $\Sigma(F)$ induces a decomposition 
$$V_\mathbb{C}=\bigoplus_{\sigma\in \Sigma(F)}2W_{\sigma},$$
where each $W_\sigma$ is a $2m$-dimensional irreducible orthogonal (respectively symplectic) $H_\mathbb{C}$-representation, where $m$ is odd. The rest of the proof then proceeds identically to the proof of Proposition \ref{rational}.

\end{proof}


\begin{prop}\label{2quat}
Let $V$ be a simple polarizable $\mathbb{Q}$-Hodge structure of weight $w\ge1$ with Hodge numbers $(n,0,\ldots,0,n)$ and endomorphism algebra $L$ of $V$ of Type II or III such that $[L:\mathbb{Q}]=n$. Let $B\cong M_2(L^{\mathrm{op}})$ be the centralizer of $L$ in $\mathrm{End}_\mathbb{Q}(V)$ and let $F$ be the center of $L$. Then
\begin{enumerate}
\item If $L$ is of Type II, then:
\begin{equation*}
Hg(V)=
\begin{cases}
R_{F/\mathbb{Q}}Sp(B,^{-}) &\mbox{if } w \text{ is odd}\\
R_{F/\mathbb{Q}}O^+(B,^{-}) &\mbox{if } w \text{ is even.}
\end{cases}
\end{equation*}
\item If $L$ is of Type III, then:
\begin{equation*}
Hg(V)=
\begin{cases}
R_{F/\mathbb{Q}}O^+(B,^{-}) &\mbox{if } w \text{ is odd}\\
R_{F/\mathbb{Q}}Sp(B,^{-}) &\mbox{if } w \text{ is even.}
\end{cases}
\end{equation*}
\end{enumerate}
\end{prop}

\begin{proof}
By identical arguments as those in Proposition \ref{oddquat}, the Hodge group $H$ is semisimple and the decomposition $L \otimes _\mathbb{Q} \mathbb{C}\cong \prod M_2(\mathbb{C})$ indexed by the set of embeddings $\Sigma(F)$ induces a decomposition 
$$V_\mathbb{C}=\bigoplus_{\sigma\in \Sigma(F)}2W_{\sigma},$$
where each $W_\sigma$ is a $4$-dimensional irreducible orthogonal (respectively symplectic) $H_\mathbb{C}$-representation. The rest of the proof then proceeds identically to the proof of Proposition \ref{2TypeI}.
\end{proof}


\begin{prop}\label{4quat}
Let $V$ be a simple polarizable $\mathbb{Q}$-Hodge structure of weight $w\ge1$ with Hodge numbers $(n,0,\ldots,0,n)$, where $n$ is four times an odd number, such that the endomorphism algebra $L$ of $V$ is a quaternion algebra over $\mathbb{Q}$. Let $B\cong M_{n/2}(L^{\mathrm{op}})$ be the centralizer of $L$ in $\mathrm{End}_\mathbb{Q}(V)$. Then:
\begin{enumerate}
\item If $L$ is of Type II, then:
\begin{equation*}
Hg(V)=
\begin{cases}
Sp(B,^{-})&\mbox{if } w \text{ is odd}\\
O^+(B,^{-}) \text{ or } SU(2^k) \text{ } \left( \text{ for } k\ge3 \text{ and }  n={{2^k}\choose {2^{k-1}}}\right) &\mbox{if } w \text{ is even.}
\end{cases}
\end{equation*}
\item If $L$ is of Type III, then:
\begin{equation*}
Hg(V)=
\begin{cases}
O^+(B,^{-}) \text{ or } SU(2^k) \text{ } \left( \text{ for } k\ge3 \text{ and } n={{2^k}\choose {2^{k-1}}}\right) &\mbox{if } w \text{ is odd}\\
Sp(B,^{-})&\mbox{if } w \text{ is even.}
\end{cases}
\end{equation*}
\end{enumerate}
\end{prop}
\begin{proof}
Again, we proceed as in the first half of the proof of Proposition \ref{oddquat}, using Lemma \ref{irreducible} to get an irreducible $n$-dimensional orthogonal (respectively symplectic) representation $W$ of the semisimple group $H_\mathbb{C}$ such that $V_\mathbb{C}=W\oplus W$. Since $n=2k$, where $k$ is, by hypothesis, twice an odd number, the rest of the proof then proceeds identically to the proof of Proposition \ref{twiceodd}.
\end{proof}

\section{Hodge Groups for Type IV Endomorphism Algebras}\label{typeiv}
Let us briefly recall the notation introduced in Section \ref{endclass} about the Type IV endomorphism algebra case. Let $V$ be a simple polarizable $\mathbb{Q}$-Hodge structure of weight $w\ge 1$ with Hodge numbers $(n,0,\ldots,0,n)$ whose endomorphism algebra $L$ is of Type IV in Albert's classification. Then the center $F_0$ of $L$ is a CM field (see Section \ref{endalgsection}) with maximal totally real subfield denoted by $F$.

Let us write $g=[F:\mathbb{Q}]$, $2n=m[L:\mathbb{Q}]$, and $q^2=[L:F_0]$. Now if $\sigma(F_0)=\{\sigma_1,\ldots,\sigma_g,\overline{\sigma}_1,\ldots,\overline{\sigma}_g\}$ is the set of embeddings of $F_0$ into $\mathbb{C}$, then $L\otimes _\mathbb{Q}\mathbb{C}$ is isomorphic to $2g$ copies of $M_q(\mathbb{C})$, one for each embedding $\sigma\in \Sigma(F_0)$, and this decomposition of $L\otimes _\mathbb{Q}\mathbb{C}$ yields a decomposition of $V^{w,0}\subset V_\mathbb{C}$ into summands $V^{w,0}(\sigma)$ on which $F_0$ acts via the embedding $\sigma \in \Sigma(F_0)$.  Letting $n_\sigma$ denote the complex dimension of $V^{w,0}(\sigma)$, we then have that $n_{\sigma_i}+n_{\overline{\sigma}_i}=mq$ for $i=1,\ldots,g$.

\begin{prop}\label{imaginary}
Let  $V$ be a simple polarizable $\mathbb{Q}$-Hodge structure of weight $w\ge1$ with Hodge numbers $(n,0,\ldots, 0,n)$ and endomorphism algebra $L$ an imaginary quadratic field.  Letting $\Sigma(L)=\{\sigma,\overline{\sigma}\}$ be the set of embeddings of $L$ into $\mathbb{C}$ and $B\cong M_n(L^{\mathrm{op}})$ be the centralizer of $L$ in $\mathrm{End}_\mathbb{Q}(V)$, if $n_\sigma$ and $n_{\overline{\sigma}}$, as defined above, are coprime, then $Hg(V)=U(B,^-).$
\end{prop}

\begin{proof}
The above result and its proof are analogous to those given in \cite[Theorem 3]{ribet1} for simple complex abelian varieties. Let $H=Hg(V)$ be the Hodge group of $V$. From  Remark \ref{lefrem}  and Table \ref{exphg}, we know $H\subseteq U(B,^-)$, so we just need to show this containment is an equality. 

As introduced in Section \ref{section5} in Equation (\ref{DL1}), the action of $L\otimes_\mathbb{Q}\mathbb{C}$ induces a decomposition
\begin{equation}\label{vecdecomp}V_\mathbb{C}=W_\sigma \oplus W_\sigma^*.\end{equation}

Let $M=MT(V)$ be the Mumford-Tate group of $V$ and consider the map
\[\overline{\rho}: M_\mathbb{C}\rightarrow \mathrm{GL}(W_\sigma)\]
induced by the action of $M_\mathbb{C}$ on $W_\sigma$. Let $N$ be the image of $\overline{\rho}$. 

Note firstly of all that $N$ is a reductive, connected subgroup of $\mathrm{GL}(W_\sigma)$. Secondly, because $L=\mathrm{End}_M(V)$, where $L\otimes _{\mathbb{Q}}\mathbb{C}\cong \mathbb{C}\oplus \mathbb{C},$
and because the action of $M$ is compatible with the decomposition in (\ref{vecdecomp}), we have must have 
$\mathrm{End}_NW_\sigma=\mathbb{C}.$ Thirdly, note that for $z\in \mathbb{C}^*$ the composition 
\[\overline{\rho} \circ h\circ \mu: \mathbb{G}_m\rightarrow \mathrm{GL}(W_\sigma)\]
acts as multiplication by $z^{-w}$ on $V^{w,0}(\sigma)$ and as the identity on $V^{0,w}(\sigma)$. Namely, $N$ contains the group of automorphisms of $W_\sigma$ that are a homothety on $V^{w,0}(\sigma)$ and the identity on $V^{0,w}(\sigma)$. Fourthly, the dimensions $n_\sigma=\dim V^{w,0}(\sigma)$ and $n_{\overline{\sigma}}=\dim V^{0,w}(\sigma)$ are coprime.

These four observations are exactly the situation of a result of Serre \cite[Proposition 5]{serre} which establishes that, under these circumstances, we have $N=\mathrm{GL}(W_\sigma).$ In particular, the fact that $\overline{\rho}$ surjects onto $\mathrm{GL}(W_\sigma)$ implies that the commutator subgroup of $M_\mathbb{C}$ surjects onto $\mathrm{SL}(W_\sigma)$.

The Lefschetz group $U(B,^-)$ of $V$ is a $\mathbb{Q}$-form of $\mathrm{GL}(n)$. Hence, by dimension arguments, in order to show that the Hodge group $H$ is equal to $U(B,^-)$, it is enough to show that the center of $M$ has dimension at least 2. To do this, it is enough to produce a two-dimensional torus which is a quotient of $M$. 

The inclusion $M\rightarrow \mathrm{GL}(_{L}V)$,
recalling that $_LV$ denotes $V$ considered as an $L$-vector space, yields a determinant map 
 \[\delta: M\rightarrow R_{L/\mathbb{Q}}\mathbb{G}_m.\]
 
 Now, the character group of $R_{L/\mathbb{Q}}\mathbb{G}_m$ is just the free abelian group on $\sigma$ and $\overline{\sigma}$. Thus 
\[(R_{L/\mathbb{Q}}\mathbb{G}_m)_{\mathbb{C}}\cong \mathbb{G}_m\times \mathbb{G}_m.\]
The Mumford-Tate group $M$ contains the torus of homotheties of $\mathrm{GL}(V)$. In the image over $\mathbb{C}$ of the map $\delta$, this torus of homotheties corresponds to the diagonal $\mathbb{G}_m$ in the expression $\mathbb{G}_m\times \mathbb{G}_m$. 

The composition 
\[\delta \circ \mu: \mathbb{G}_m \rightarrow \mathbb{G}_m\times \mathbb{G}_m\]
is given by
\[z\in \mathbb{C}^*\mapsto(z^{n_\sigma},z^{n_{\overline{\sigma}}}).\]
But since $n_\sigma$ and $n_{\overline{\sigma}}$ are coprime, they are not equal and so the image of $\delta \circ \mu$ is a subtorus of $\mathbb{G}_m\times \mathbb{G}_m$ which is not the diagonal. Hence $\delta$ is surjective. So, the Mumford-Tate group $M$ indeed has a  quotient which is a two-dimensional torus and thus we have $H=U(B,^-)$.\end{proof}

\subsection{$E$-Hodge Structures}

Let $E$ be a number field. Define an $E$-\emph{Hodge structure} to be a $\mathbb{Q}$-Hodge structure $V$ together with a homomorphism of $\mathbb{Q}$-algebras $E\rightarrow \mathrm{End}_{\mathbb{Q}-HS}(V).$

Suppose $E$ is, in fact, a totally real or CM field. Then, writing $a\to \overline{a}$ for the involution on $E$ given by complex conjugation (which is the identity involution if $E$ is totally real), we may define a \emph{polarized} $E$\emph{-Hodge structure} to be a polarized $\mathbb{Q}$-Hodge structure together with a homomorphism $E\rightarrow \mathrm{End}_{\mathbb{Q}-HS}(V)$ of $\mathbb{Q}$-algebras with involution. Namely, the form $\langle,\rangle:V\times V\rightarrow \mathbb{Q}$ satisfies
$\langle ax,y\rangle=\langle x,\overline{a}y\rangle$
for all $a\in E$ and all $x,y\in V$. 
In fact, if $V$ is an $E$-Hodge structure whose underlying $\mathbb{Q}$-Hodge structure is polarizable, then $V$ is polarizable as an $E$-Hodge structure \cite[Lemma 2.1]{totaro}. There does not seem to be a good definition of a polarized $E$-Hodge structure for $E$ a number field which is not totally real or a CM field. 

If $V$ is an $E$-Hodge structure of weight $w$, then each $V^{p,q}$ in the decomposition of $V\otimes_\mathbb{Q}\mathbb{C}$ is an $E\otimes _\mathbb{Q}\mathbb{C}$-module and so splits as a direct sum
$$V^{p,q}=\bigoplus_{\sigma \in \Sigma(E)}V^{p,q}(\sigma),$$
where $V^{p,q}(\sigma)$ is the subspace of $V\otimes _\mathbb{Q} \mathbb{C}$ where $E$ acts via $\sigma$. 

We say that $V$ has \emph{Hodge numbers} $(a_0,\ldots,a_w)$ as an $E$-Hodge structure if, for each embedding $\sigma\in \Sigma(E)$, the summand $V^{j,w-j}(\sigma)$ has complex dimension $a_j$ for all $j$. Note that if $V$ is an $E$-Hodge structure with Hodge numbers $(a_0,\ldots,a_w)$ and $[E:\mathbb{Q}]=r$, then $V$ is a $\mathbb{Q}$-Hodge structure with Hodge numbers $(ra_0,\ldots,ra_w)$.

\begin{lem}\label{sunitary}
Let $V$ be a polarizable $\mathbb{Q}$-Hodge structure of weight $w\ge1$ with Hodge numbers $(n,0,\ldots, 0,n)$. Suppose there exists a CM field $E$ embedding into the endomorphism algebra $L$ of $V$. Writing $[E:\mathbb{Q}]=r$ and $2n=lr$, let $J$ be the maximal totally real subfield of $E$ and let $C\cong M_{l}(E^{\mathrm{op}})$ be the centralizer of $E$ in $\mathrm{End}_\mathbb{Q}(V)$. Then we have the inclusion \[Hg(V)\subseteq R_{J/\mathbb{Q}}SU(C,^-)\]
if and only if $V$ is an $E$-Hodge structure with Hodge numbers $(\frac{n}{r},0,\ldots,0,\frac{n}{r})$.
\end{lem}
\begin{proof}
We know $V\otimes _\mathbb{Q} \mathbb{C}$ has rank $l$ as a free $E\otimes _\mathbb{Q} \mathbb{C}$-module, so consider the exterior product 
$\bigwedge^{l}_E V.$
An element $\alpha$ in $Hg(V)$ acts on $\bigwedge^{l}_E V$ as multiplication by $\mathrm{Norm}_{E}(\alpha)$. Since $E$ is contained in $L$, we already know that 
$Hg(V)$ is contained in $R_{J/\mathbb{Q}}U(C,^-)$. Hence we have $Hg(V)\subseteq R_{J/\mathbb{Q}}SU(C,^-)$
if and only if $\bigwedge^{l}_E V$ is invariant under the action of $Hg(V)$, meaning if and only if $\bigwedge^{l}_E V$ is purely of type $(\frac{l}{2},\frac{l}{2})$ as a sub-Hodge structure of $\bigwedge^{l}_\mathbb{Q} V$.

For each $\sigma\in \Sigma(E)$, let $e_\sigma=\dim _\mathbb{C}V^{w,0}(\sigma)$ in the decomposition 
\[V\otimes _{\mathbb{Q}}\mathbb{C}=\bigoplus_{\sigma\in \Sigma(E)} \left(V ^{w,0}(\sigma)\oplus V^{0,w}(\sigma)\right).\]

Then we have, 
\begin{equation*}
\begin{split}
\left(\bigwedge^{l}_E V\right)\otimes_{\mathbb{Q}} \mathbb{C} & =\bigwedge^{l}_{E\otimes _{\mathbb{Q}}\mathbb{C}} (V\otimes _{\mathbb{Q}}\mathbb{C})\\
&=\bigoplus_{\sigma\in \Sigma(E)}\bigwedge^{l}\left(V ^{w,0}(\sigma)\oplus V^{0,w}(\sigma)\right)\\
&=\bigoplus_{\sigma\in \Sigma(E)}\bigoplus_{e_\sigma}\left(\bigwedge^{e_\sigma}V ^{w,0}(\sigma)\otimes \bigwedge^{l-e_\sigma}V^{0,w}(\sigma)\right).
\end{split}
\end{equation*}

Thus $\bigwedge^{l}_E V$ is purely of type $(\frac{l}{2},\frac{l}{2})$ if and only if $e_\sigma=\frac{l}{2}=l-e_\sigma$ for all embeddings $\sigma\in \Sigma(E)$ of $E$ into $\mathbb{C}$. But the property of having $e_\sigma=\frac{l}{2}=l-e_\sigma$ for all embeddings $\sigma\in \Sigma(E)$ is exactly what it means for $V$ to be an $E$-Hodge structure with Hodge numbers $(\frac{n}{r},0,\ldots,0,\frac{n}{r})$, so this finishes the proof.

\end{proof}


\begin{prop}\label{Ealg}
Let $V$ be a simple polarizable $E$-Hodge structure with Hodge numbers $(p,0,\ldots,0,p)$, where $p$ is prime and $E$ is a CM field of degree $r$ over $\mathbb{Q}$. Suppose the endomorphism algebra of $V$ as a $\mathbb{Q}$-Hodge structure is $E$. Let $C\cong M_{2p}(E^{\mathrm{op}})$ be the centralizer of $E$ in $\mathrm{End}_\mathbb{Q}(V)$ and let $J$ be the maximal totally real subfield of $E$. Then $Hg(V)=R_{J/\mathbb{Q}}SU(C,^-).$
\end{prop}

\begin{proof}
Let $H=Hg(V)$ be the Hodge group of $V$.  Since $V$ is an $E$-Hodge structure, by Lemma \ref{sunitary} we know
$H\subseteq R_{J/\mathbb{Q}}SU(C,^{-}),$ where $SU(C,^{-})$ is a $J$-form of $SL(2p)$. We just need to prove equality.

Since $E=[\mathrm{End}_\mathbb{Q}(V)]^H$, the center of $H$ is contained in the center of $E$, which in this case is all of $E$.  Hence the center of $H$ is contained in  $R_{J/\mathbb{Q}}SU(C,^-)\cap E$. So elements of the center of $H$ must be $2n$-th roots of unity in $E$, of which there are only finitely many. Thus, the center of $H$ is finite, which, since $H$ is also reductive, implies that $H$ is semisimple.

In the notation of Section \ref{section5}, let 
\[W_\sigma=\rho_1\boxtimes\cdots \boxtimes \rho_s\] be an irreducible representation of $\mathrm{Lie}(H)_\mathbb{C}$ in the decomposition 
\[V_\mathbb{C}=\bigoplus_{\sigma\in \Sigma(J)}(W_\sigma\oplus W_\sigma^*).\]
Namely the representation $W_\sigma$ is $2p$-dimensional and not self-dual. So, we may assume without loss of generality than none of the $\rho_j$ are self-dual either. Lemma \ref{tablelem} yields that the highest weight of each of the $\rho_j$ is minuscule and each $\mathfrak{g}_j$ is of classical type. Hence consulting Table \ref{selfdual} yields the following possibilities for the Lie algebra $\mathfrak{g}_1\times \cdots \times \mathfrak{g}_s$ acting on $W_i$:
\begin{enumerate}
\item The Lie algebra $\mathfrak{sl}_{2p}$ acting by the standard representation
\item When $p=5$, the Lie algebra $\mathfrak{sl}_5$ acting by $\bigwedge^2(\mathrm{Standard})$
\item When $p\ne 2$, the Lie algebra $\mathfrak{sl}_2\times \mathfrak{sl}_p$ acting by the product of the standard representations.
\end{enumerate}

Now let $M_\mathbb{C}$ be the complexified Mumford-Tate group of $V$ and consider Case (2). Write the weights in $W_i$ of the homomorphism $\gamma\colon \mathbb{G}_{m,\mathbb{C}}\rightarrow M_\mathbb{C}$ as $z\mapsto (z^{-c}, z^{-l_1}),$ according to the decomposition $W=\chi\boxtimes\rho_1$, where $\chi$ is a character of $\mathfrak{c}$ and $\rho_1$ is the representation of $\mathfrak{sl}_5$ acting by $\bigwedge^2(\mathrm{Standard})$. 

Let $\lambda_1,\ldots,\lambda_5$ be the $5$ possible values of $l_1$. Then we know, half of the elements in the set $S:=\{\lambda_i+\lambda_j\mid 1\le i<j\le 5\}$ are equal to $0$ and half are equal to $w-c$. Using a pigeon-hole argument, there exists $i$  and  $j,k\ne i$ such that $\lambda_i=-\lambda_j=-\lambda_k$. But since either $\lambda_j=-\lambda_k$ or $\lambda_j=w-c-\lambda_k$, we must have either $\lambda_i=\lambda_j=\lambda_k=0$ or $\lambda_i=-\frac{w-c}{2}$ and $\lambda_j=\lambda_k=\frac{w-c}{2}$. In either case though,  there is no way to choose the two remaining values of $l_1$ so as to satisfy the requirement on the set $S$. So the Case (2) in the above list not possible. 

Similarly, for Case (3) on the list, write the weights in $W_i$ of the homomorphism $\gamma\colon \mathbb{G}_{m,\mathbb{C}}\rightarrow M_\mathbb{C}$ as $z\mapsto (z^{-c}, z^{-l_1}, z^{-l_2})$ according to the decomposition $W=\chi\boxtimes\rho_1\boxtimes \rho_2$, where $\rho_1$ is the standard representation of $\mathfrak{sl}_2$ and $\rho_2$ is the standard representation of $\mathfrak{sl}_p$. Denote the two possible values of $l_1$ by $\alpha$ and $\beta$ and denote the possible values of $l_2$ by $\lambda_1,\ldots, \lambda_p$. Then by Lemma \ref{sl2lem}, either $\alpha=\beta=0$ or $\lambda_1=\cdots=\lambda_p=0$. But if $\alpha=\beta=0$, then half of the $\lambda_k$ for must be equal to $w-c$ and half of the $\lambda_k$ must be equal to $0$. Since $p$ is odd, this is impossible. Therefore, we must have $\lambda_1=\cdots=\lambda_p=0$ and hence $\alpha=0$ and $\beta=w-c$. But then the homomorphism $\gamma$ factors through $Z(M)_\mathbb{C}\cdot SL(2,\mathbb{C})^{r/2}$. So then $H_\mathbb{C}\subset SL(2,\mathbb{C})^{r/2}$. But since $\mathfrak{sl}_2\times \mathfrak{sl}_p$ acts on the representation $W$ by the product of the standard representations, the group $H_\mathbb{C}$ must surject onto $SL(p,\mathbb{C})$, which contradicts the statement that $H_\mathbb{C}$ is contained in $SL(2,\mathbb{C})^{r/2}$. Hence, Case (3) is not possible. 

Hence the representation $W_\sigma$ must be the standard representation of $\mathfrak{sl}_{2p}$.  Applying Lemma  \ref{goursat}, then gives $H_\mathbb{C}= \prod_{\sigma\in \Sigma(J)}SL(2p)$, which completes the proof.

\end{proof}


\begin{prop}\label{Ecoprime}
Let $V$ be a simple polarizable $\mathbb{Q}$-Hodge structure of weight $w\ge1$ which is an $E$-Hodge structure with Hodge numbers $(n,0,\ldots,0,n)$ for some CM field $E$. Writing $[E:\mathbb{Q}]=r$, suppose that the endomorphism algebra $L$ of $V$ is a CM field such that $[L:\mathbb{Q}]=2r$. Moreover suppose that $n_{\sigma}=\dim V^{w,0}(\sigma)$ and $n_{\overline{\sigma}}=\dim V^{w,0}(\overline{\sigma})$ are coprime for all embeddings $\sigma \in \Sigma(L)$. Then, letting $B\cong M_n(L^{\mathrm{op}})$ denote the centralizer of $L$ in $\mathrm{End}_\mathbb{Q}(V)$ and $F$ the maximal totally real subfield of $L$, we have
\[Hg(V)=R_{F/\mathbb{Q}}SU(B,^-).\]
\end{prop}

\begin{proof}
Let $M=MT(V)$ be the Mumford-Tate group of $V$ and, as in the proof of Proposition \ref{imaginary}, consider the map
\[\overline{\rho}: M_\mathbb{C}\rightarrow \mathrm{GL}(W_\sigma)\]
induced by the action of $M_\mathbb{C}$ on an $n$-dimensional summand $W_\sigma$ in the decomposition $V_\mathbb{C}=\bigoplus_{\sigma \in \Sigma(F)}W_\sigma \oplus W_\sigma^*$.  Since $L\otimes _\mathbb{Q}\mathbb{C}\cong \prod_{\sigma\in \Sigma(F)} (\mathbb{C} \oplus \mathbb{C})$, if $N$ is the image of $\overline{\rho}$, we have $\mathrm{End}_N(W_\sigma)=\mathbb{C}$. So by an argument identical to that given in the proof of Proposition \ref{typeiv}, the map $\overline{\rho}$ is surjective. Hence the commutator subgroup of $M_\mathbb{C}$ surjects onto $\mathrm{SL}(W_\sigma)$.

Now, since the Hodge group $Hg(V)$ is contained in the Lefschetz group we know
\begin{equation}\label{suscontain}R_{F/\mathbb{Q}}SU(B,^-)\subseteq H\subseteq R_{F/\mathbb{Q}}U(B,^-).\end{equation}
Note that $U(B,^-)$ is an $F$-form of $GL(n)$ and $SU(B,^-)$ is an $F$-form of $SL(n)$.  

But since $V$ is an $E$-Hodge structure with Hodge numbers $(n,0,\ldots,0,n)$, by Lemma \ref{sunitary} we also know
 \[H\subseteq R_{J/\mathbb{Q}}SU(C,^{-}),\]
 where $J$ is the subfield of $E$ fixed by complex conjugation and $C\cong M_{2n}(E^\mathrm{op})$ is the centralizer of $E$ in $\mathrm{End}_\mathbb{Q}(V)$. Note that the group $SU(C,^{-})$ is a $J$-form of $SL(2n)$.

We know $L=[\mathrm{End}_\mathbb{Q}(V)]^H$, hence the center of $H$ is contained in $L$. Hence the center of $H$ is contained in  $R_{F/\mathbb{Q}}U(B,^-)\cap L$. So if $\lambda_1$ and $\lambda_2$ are elements of the center of $H$, both must have norm $1$. Moreover, because $H$ is contained in $R_{J/\mathbb{Q}}SU(C,^{-})$, considering an embedding of $R_{F/\mathbb{Q}}U(B,^-)\cap L$ into $R_{J/\mathbb{Q}}SU(C,^{-})$, we must have $\lambda_1\lambda_2$ an $n$-th root of unity in $L$. There are only finitely many such $\lambda_1$ and $\lambda_2$ in $L$, so the center of $H$ is finite, which, since $H$ is reductive, implies $H$ is semisimple.  The inclusions in (\ref{suscontain}) thus yield $H=R_{F/\mathbb{Q}}SU(B,^-)$.
\end{proof}

\begin{prop}\label{index2}
Let $V$ be a simple polarizable $\mathbb{Q}$-Hodge structure of weight $w\ge1$ which is an $E$-Hodge structure with Hodge numbers $(n,0,\ldots,0,n)$ for some CM field $E$. Writing $[E:\mathbb{Q}]=r$, suppose that the endomorphism algebra $L$ of $V$ is a CM field such that $[L:\mathbb{Q}]=nr$. Then, letting $B\cong M_2(L^{\mathrm{op}})$ be the centralizer of $L$ in $\mathrm{End}_\mathbb{Q}(V)$ and $F$  the  maximal totally real subfield of $L$, we have
\[Hg(V)=R_{F/\mathbb{Q}}SU(B,^-).\]
\end{prop}

\begin{proof}
Let $H=Hg(V)$. The Lefschetz group of $V$ is $R_{F/\mathbb{Q}}U(B,^-)$, so we have:
\begin{equation}\label{ucontain}H\subseteq R_{F/\mathbb{Q}}U(B,^-),\end{equation}
where $U(B,^-)$ is an $F$-form of $GL(2)$. 

Since $V$ is an $E$-Hodge structure with Hodge numbers $(n,0,\ldots,0,n)$. Thus by Lemma \ref{sunitary} we have:
 $$H\subseteq R_{J/\mathbb{Q}}SU(C,^{-}),$$ where $J$ is the maximal totally real subfield of $E$ and $C\cong M_{2n}(E^\mathrm{op})$ is the centralizer of $E$ in $\mathrm{End}_\mathbb{Q}(V)$. Here, the group $SU(C,^{-})$ is a $J$-form of $SL(2n)$.

By the same argument as in the proof of Proposition \ref{Ecoprime}, the center of $H$ is contained in $R_{F/\mathbb{Q}}U(B,^-)\cap L$ and so all its elements have norm $1$ and since $H$ is contained in $R_{J/\mathbb{Q}}SU(C,^{-})$ any product of $n$ such elements is equal to $\pm1$. Since there are only finitely many elements of $L$ with this property, the center of $H$ is finite and so, since $H$ is reductive, we have that $H$ is semisimple. The inclusion of  (\ref{ucontain}) then yields
$$H\subseteq R_{F/\mathbb{Q}}SU(B,^-),$$ 
where $SU(B,^-)$ is an $F$-form of $SL(2)$.  Then in the decomposition $V_\mathbb{C}=\bigoplus_{\sigma\in \Sigma(F)}W_\sigma\oplus W_\sigma^*$ coming from the action of $L\otimes_\mathbb{Q}\mathbb{C}$, the $2$-dimensional irreducible representation $W_\sigma$ of $\mathrm{Lie}(H)_\mathbb{C}$ is just the standard representation of $\mathfrak{sl}_2$.  By Lemma \ref{goursat}, we get $\mathrm{Lie}(H)_\mathbb{C}=\prod_{\sigma\in \Sigma(F)}\mathfrak{sl}_2$, which finishes the proof. 
\end{proof}

Consider the torus given by
\[U_L=\ker( \mathrm{Norm}_{L}\colon R_{L/\mathbb{Q}}\mathbb{G}_m\rightarrow R_{F/\mathbb{Q}}\mathbb{G}_m).\]

Letting $J$ denote the maximal totally real subfield of $E$,  define the torus $SU_{L/E}$ to be the subtorus of $U_L$ given by
\[SU_{L/E}=\ker( \mathrm{Norm}_{L/E}\colon U_L\rightarrow R_{J/\mathbb{Q}}\mathbb{G}_m).\]
Observe that $SU_{L/E}$ has codimension $[J:\mathbb{Q}]$ in $U_L$.

\begin{prop}\label{contained}
Let $V$ be a simple polarizable $\mathbb{Q}$-Hodge structure of weight $w\ge1$ which is an $E$-Hodge structure with Hodge numbers $(n,0,\ldots,0,n)$ for some CM field $E$. Writing $[E:\mathbb{Q}]=r$, suppose that the endomorphism algebra $L$ of $V$ is a CM field and $[L:\mathbb{Q}]=2nr$. Then we have the inclusion $Hg(V)\subseteq SU_{L/E}.$
\end{prop}

\begin{proof}
Since $\dim_LV=1$, the Lefschetz group of $V$ is just the torus $U_L$ and thus we know $Hg(V)\subseteq U_L$. Now, as in previous proofs, by Lemma \ref{sunitary}, we know $Hg(V)\subseteq R_{J/\mathbb{Q}}SU(C,^{-})$. Recall that 
$SU(C,^{-})$ is defined to be the kernel of the norm map $\mathrm{Norm}_{C/E}$ acting on $U(C,^{-})$, and so, in particular, given the inclusion $Hg(V)\subseteq U_L$, all the elements of $Hg(V)$ lie in the kernel of the norm map $\mathrm{Norm}_{L/E}$, which implies $Hg(V)\subseteq SU_{L/E}.$
\end{proof}

\begin{prop}\label{Etorus} Let $V$ be a simple polarizable $\mathbb{Q}$-Hodge structure of weight $w\ge1$  which is an $E$-Hodge structure with Hodge numbers $(p,0,\ldots,0,p)$, where $p$ is an odd prime, for some imaginary quadratic field $E$. Suppose that the endomorphism algebra $L$ of $V$ satisfies:
\begin{enumerate}
\item $L$ is of Type IV
\item $L/\mathbb{Q}$ is a Galois extension
\item $[L:\mathbb{Q}]=4p$
\end{enumerate}
Then $L$ is a CM field and 
$Hg(V)=SU_{L/E}.$
\end{prop}

\begin{proof}
Note, first of all, that using Albert's classification, since $L$ is of Type IV, we know the center $F_0$ of $L$ is a CM field and hence $[F_0:\mathbb{Q}]$ is even. Writing $q^2=[L:F_0]$, since $[L:\mathbb{Q}]=4p$ with $p$ odd, we have that $q=1$ and so $L$ is a CM field. Thus, Proposition \ref{contained} implies $Hg(V)\subseteq SU_{L/E}.$

Since $E$ is an imaginary quadratic field, we know this torus $SU_{L/E}$ has codimension $1$ in $U_L$, which has rank $[F:\mathbb{Q}]=2p$. Hence the rank of $SU_{L/E}$ is $2p-1$. 

Following \cite{dodson} and \cite{kubota}, for $\Sigma(L)=\{\sigma_1,\ldots, \sigma_{4p}\}$ the set of embeddings of $L$ into $\mathbb{C}$, a \emph{CM type} $\Theta\subset \Sigma(L)$ is defined by the criterion
$\Theta \cup \overline{\Theta}=\Sigma(L)$ and 
the Hodge structure $V$ corresponds to a unique CM type $\Theta$ on L \cite[Section V.C]{domain}.

Now let $\tilde{L}$ be the Galois closure of $L$ and consider $\mathrm{Gal}(\tilde{L}/\mathbb{Q})$. For every $g\in \mathrm{Gal}(\tilde{L}/\mathbb{Q})$ and for every $\sigma_i\in \Theta$, let $\sigma_i^g\colon L \rightarrow \mathbb{C}$ be the element of $\Sigma(L)$ defined by $x \mapsto g\cdot \sigma_i(x)$. We then have a CM-type on $L$ given by
$$\Theta^g=\{\sigma_1^g,\ldots, \sigma_{4p}^g\}.$$

The \emph{Kubota rank} of $\Theta$, denoted $\mathrm{Rank}(\Theta)$, is the rank over $\mathbb{Z}$ of the submodule of $\mathbb{Z}[\Sigma(L)]$ spanned by the set
$\{ \Theta^g \mid g\in \mathrm{Gal}(\tilde{L}/\mathbb{Q})\}.$

We then have \cite[Proposition V.D.5]{domain}
\begin{equation}\label{cmrank}
\mathrm{Rank}(\Theta)=\dim_\mathbb{Q}Hg(V).
\end{equation}

However, Tankeev proves in \cite[Corollary 3.15]{tankeev}, that for a simple CM type $\Theta$ on a CM field $L$ of degree $2p$ over $\mathbb{Q}$, where $p$ is an odd prime,  we have:
$$\mathrm{Rank}(\Theta) \ge 2p-1.$$

Hence it follows from (\ref{cmrank}) that we have $\dim_\mathbb{Q}Hg(V)\ge 2p-1$ and so the result is proved. 
\end{proof}

We now combine the  above results to obtain the following theorem about simple polarizable $\mathbb{Q}$-Hodge structures that have endomorphism algebra of Type IV and that are also $E$-Hodge structures with Hodge numbers $(p,0,\ldots,0,p)$. 

\begin{thm}\label{2pthm}
Let $V$ be a simple polarizable $\mathbb{Q}$-Hodge structure of weight $w\ge1$ which is an $E$-Hodge structure with Hodge numbers $(p,0,\ldots,0,p)$, where $p$ is a prime and $E$ is an imaginary quadratic field.  If the endomorphism algebra $L=\mathrm{End}_{\mathbb{Q}-HS}(V)$ is of Type IV, then $L$ is a CM field. Writing $4p=m{[L:\mathbb{Q}]}$, letting $B\cong M_m(L^{\mathrm{op}})$ be the centralizer of $L$ in $\mathrm{End}_\mathbb{Q}(V)$, and letting $F$ be the maximal totally real subfield of $L$, we have
\begin{equation*}
Hg(V)=
\begin{cases}
R_{F/\mathbb{Q}}SU(B,^-) &\mbox{if } [L:\mathbb{Q}]\ne 4p\\
SU_{L/E} &\mbox{if } [L:\mathbb{Q}]=4p \text{ and either } p=2 \text{ or } L/\mathbb{Q} \text{ is a Galois extension.}
\end{cases}
\end{equation*}
\end{thm}
\begin{proof}
Observe that since $\dim_EV=2p$, we must have, using the embedding $E\hookrightarrow L$, that $[L:E]$ is equal to one of $1$, $2$, $p$, or $2p$.  The case when $[L:E]=1$ is taken care of by Proposition \ref{Ealg}.  

In the case when $[L:E]=2$, we know $[L:\mathbb{Q}]=4$. Since $L$ is assumed to be of Type IV, Albert's classification (see Section \ref{endalgsection}) yields that $L$ is a CM field. So $m=p$ and $q=1$, and thus  we have $n_\sigma+n_{\overline{\sigma}}=p$ for all embeddings $\sigma \in \Sigma(L)$. But since $p$ is prime, the numbers $n_{\sigma}$ and $n_{\overline{\sigma}}$ are coprime and so the result follows from Proposition \ref{Ecoprime}.  

When $[L:E]=p$, we know $[L:\mathbb{Q}]=2p$, and so, because $L$ is of Type IV, Albert's classification yields that $L$ is a CM field. Hence the result follows from Proposition \ref{index2}. 

Lastly, consider the case when $[L:E]=2p$, namely when $[L:\mathbb{Q}]=4p$. If $p$ is an odd prime, then the result follows by Proposition \ref{Etorus}. So suppose we have $p=2$. Then Albert's classification yields that either $L$ is a CM field of degree $8$ over $\mathbb{Q}$ or $L$ is a division algebra of degree $4$ over an imaginary quadratic field $F_0$. In the latter case, we would have $n_\sigma+n_{\overline{\sigma}}=2$, where $\sigma$ and $\overline{\sigma}$ are the $2$ embeddings of $F_0$ into $\mathbb{C}$. Namely, we have either $n_\sigma=n_{\overline{\sigma}}=1$ or $n_\sigma=2$ and $n_{\overline{\sigma}}=0$, which correspond to exceptional cases (3) and (4) in Totaro's classification (see Theorem \ref{totthm}). Hence the endomorphism algebra $L$ cannot be a division algebra and therefore must be a CM field of degree $8$ over $\mathbb{Q}$. Moreover, by Lemma \ref{sunitary}, because $V$ is an $E$-Hodge structure, we know $Hg(V)$ is contained in $SU_{L/E}$, which has dimension $3$. However by Lemma \ref{bound}, the group $Hg(V)$ has dimension at least $3$. So $Hg(V)$ is equal to $SU_{L/E}$, which finishes the proof. 
\end{proof} 

\section{Main Results}\label{mainresults}

\begin{prop}\label{onedim}
Let $V$ be a simple polarizable $\mathbb{Q}$-Hodge structure of weight $w\ge1$ and Hodge numbers $(1,0,\ldots,0,1)$. Then
$Hg(V)=Lef(V).$
\end{prop}

\begin{proof}
Albert's classification yields that the endomorphism algebra $L$ of $V$ is either $\mathbb{Q}$ or an imaginary quadratic field. Writing $2=m[L:\mathbb{Q}]$, this corresponds to $L$ being of Type I with $m=2$ or $L$ being of Type IV with $m=1$. When the weight $w$ is odd, the first case corresponds to exceptional case (6) in Totaro's classification. Thus, if $w$ is odd, then the endomorphism algebra $L$ of $V$ is either $\mathbb{Q}$ or an imaginary quadratic field, while if $w$ is even, then $L$ must be an imaginary quadratic field. 

In the case when $L$ is $\mathbb{Q}$, then by Table  \ref{exphg}, we have $Hg(V)\subset SL(V)$, where $SL(V)$ has rank $1$. However, by Remark \ref{triv}, the Hodge group $Hg(V)$ is nontrivial and by Remark \ref{ss} the group $Hg(V)$ is semisimple. Hence $Hg(V)$ is equal to $SL(2)$. 

In the case when $L$ is an imaginary quadratic field, Table  \ref{exphg} yields  $Hg(V)\subset U_L$. Since $U_L$ has dimension $1$ and $Hg(V)$ must be nontrivial, we get $Hg(V)=U_L$.
\end{proof}

\subsection{Hodge Numbers $(p,0,\ldots,0,p)$}
\begin{thm}\label{primethm} Let $V$ be a simple polarizable $\mathbb{Q}$-Hodge structure of weight $w\ge1$ with Hodge numbers $(p,0,\ldots, 0, p)$, where $p$ is prime. Then
$Hg(V)=Lef(V).$
\end{thm}

\begin{proof}
When $V$ is of odd weight, the result follows using the equivalence of Remark \ref{equiv} together with results of Ribet \cite[Theorems 1,2]{ribet1}, from whose proofs we extensively borrow, and Tankeev \cite{tankeevrib}, which prove the result for simple complex abelian varieties of prime dimension. Thus we may assume that $V$ is of even weight.

Assume first that $p$ is odd. Then Albert's classification yields that the endomorphism algebra $L$ of $V$ is either a totally real field or a CM field. Moreover, Totaro's exceptional case (6) allows us to eliminate the possibility that $L$ is a totally real field of degree $p$ over $\mathbb{Q}$. Thus we are left with the following possibilities for $L$:
\begin{enumerate}
 \item $\mathbb{Q}$ (Type I)
 \item An imaginary quadratic field (Type IV)
 \item A CM-field of degree $2p$ over $\mathbb{Q}$ (Type IV).
 \end{enumerate} 

First consider Case (1). By Table \ref{exphg}, the Lefschetz group $Lef(V)$ of $V$ is $SO(V)$. Letting $H=Hg(V)$ denote the Hodge group of $V$, Proposition \ref{rational} yields that $H$ is either equal to $Lef(V)$ or is equal to $SU(2^k)$, where $2p={2^k\choose 2^{k-1}}$ for some $k\ge 3$. Thus, to prove the result we must show that this second option cannot occur. To do this, we use the following simple combinatorial argument to show that we cannot have $2p={2^k\choose 2^{k-1}}$ for any $k\ge 3$.
We know
$${2^k\choose 2^{k-1}}=2^{2^{k-1}}\cdot\frac{1\cdot3\cdot5\cdots(2^k-1)}{(2^{k-1})!}.$$
Moreover, by De Polignac's formula for the prime factorization of $n!$, the 2-adic order of the term $(2^{k-1})!$ is $2^{k-1}-1$.  So, after cancellation, we have
\begin{equation}\label{choose}{2^k\choose 2^{k-1}}=2\cdot \frac{(\text{product of all odd numbers less than }2^k)}{(\text{product of some odd numbers all less than }2^{k-1})}.\end{equation}
Using known bounds on the prime-counting function $\pi$ \cite[Corollary 1]{primecounting} yields $\pi(2^k)-\pi(2^{k-1})\ge2$ for $k\ge3$. Namely, the numerator in (\ref{choose}) always contains at least two terms not cancelled by the denominator. Hence we cannot have $2p={2^k\choose 2^{k-1}}$ for any $k\ge3$. This finishes Case (1). 

Now consider Case (2). In this case, the Lefschetz group $Lef(V)$ of $V$ is $U(B,^{-})$. As in previous proofs, since $m=p$ and $q=1$ in this case, we have $n_\sigma+n_{\overline{\sigma}}=p$ in the decomposition of $V\otimes_\mathbb{Q}\mathbb{C}$ induced by $L\otimes _\mathbb{Q}\mathbb{C}$. Hence $n_{\sigma}$ and $n_{\overline{\sigma}}$ are coprime and thus the result follows from Proposition \ref{imaginary}. 

Finally, consider Case (3). In this case, the Lefschetz group of $V$ is  $U_L=\ker( \mathrm{Norm}_{L}\colon R_{L/\mathbb{Q}}\mathbb{G}_m\rightarrow R_{F/\mathbb{Q}}\mathbb{G}_m)$, so we know $H\subset U_L$.  Since $[F:\mathbb{Q}]=p$, the torus $U_L$ has dimension $p$ over $\mathbb{Q}$, and so to show $H=U_L$, we just need to show that $\dim_\mathbb{Q}(H)=p$.

Consider the character groups $X^*(H)$, $X^*(U_L)$, and $X^*(R_{L/\mathbb{Q}}\mathbb{G}_m)$. Since both $H$ and $U_L$ are contained in $R_{L/\mathbb{Q}}\mathbb{G}_m$, the groups $X^*(H)$ and $X^*(U_L)$ are quotients of $X^*(R_{L/\mathbb{Q}}\mathbb{G}_m)$.  Here $X^*(R_{L/\mathbb{Q}}(\mathbb{G}_m))$ is the free abelian group on the embeddings $\sigma \in \Sigma(L)$. The group $X^*(R_{L/\mathbb{Q}}\mathbb{G}_m)$ has a natural left action by $\mathrm{Gal}(\overline{\mathbb{Q}}/\mathbb{Q})$. The group $X^*(U_L)$ is the quotient of $X^*(R_{L/\mathbb{Q}}\mathbb{G}_m)$ by the relation $\sigma + \overline{\sigma}=0$ for all embeddings $\sigma \in \Sigma(L)$. The group $X^*(H)$ is a quotient of $X^*(U_L)$ with the property that the images of the embeddings $\sigma$ in $X^*(H)$ are all distinct. Indeed, since $L=\mathrm{End}_H(V)$ is commutative, the $H$-module $V$ cannot have multiplicities greater than $1$ in its decomposition into simple modules over $H$, each corresponding to an embedding $\sigma\in \Sigma(L)$. 

Now, consider the homomorphism 
\[\rho:\mathrm{Gal}(\overline{\mathbb{Q}}/\mathbb{Q})\rightarrow \mathrm{Aut}(X^*(H))\]
giving the action of $\mathrm{Gal}(\overline{\mathbb{Q}}/\mathbb{Q})$ on $\mathrm{Aut}(X^*(H))$. Because the elements of $X^*(H)$ are the embeddings $\sigma\in \Sigma(L)$, each one occurring once, we have
\[\mathrm{Ker}(\rho)=\mathrm{Gal}(\overline{\mathbb{Q}}/\tilde{L}),\]where $\tilde{L}$ denotes the Galois closure of $L$. Hence the order of $\mathrm{Im}(\rho)$ is $[\tilde{L}:\mathbb{Q}]$. But  $[\tilde{L}:\mathbb{Q}]$ is divisible by the prime number $p$, since $[L:\mathbb{Q}]=2p$. Hence we may choose some $g \in \mathrm{Im}(\rho)$ of order $p$. Since $g$ is in $\mathrm{Aut}(X^*(H))$, the action of $g$ on the $\mathbb{Q}$-vector space 
$$Y=X^*(H)\otimes \mathbb{Q}$$
makes $Y$ into a module over $\mathbb{Q}[x] / (x^p-1)$. We may write $\mathbb{Q}[x]/(x^p-1)=\mathbb{Q}(\mu_p)\times \mathbb{Q}$, where $\mu_p$ is a $p$-th root of unity. Thus write 
$$Y=Y_1\oplus Y_2,$$ 
where $Y_1$ is a $\mathbb{Q}(\mu_p)$-vector space and $Y_2$ is a $\mathbb{Q}$-vector space. Because $g$ has order $p$ and thus does not have order $1$, the element $x$ does not act as the identity on $Y_1$. So $Y_1$ is nonzero. Then the dimension of $Y_1$ over $\mathbb{Q}$ is a multiple of $p-1$, which yields
$$\dim_\mathbb{Q}(H)\ge p-1.$$

To show that $\dim_\mathbb{Q}(H)=p$, it just remains to show that $Y_2$ is nonzero.  Namely  we need to show that $Y$ contains a nonzero element fixed under the action of $g$. Choosing some embedding $\sigma_0: L \rightarrow \mathbb{C}$, it is clear that the element 
\begin{equation}\label{defchi}\chi=\sigma_0+g\sigma_0+\cdots+g^{p-1}\sigma_0\end{equation}
in $X^*(H)$ is fixed by $g$. So we just need to show that $\chi$ is nonzero.

Let $M$ be the Mumford-Tate group of $V$ and consider the cocharacter groups $X_*(H)$, $X_*(M)$, and $X_*(R_{L/\mathbb{Q}}\mathbb{G}_m)$. As in Proposition \ref{Etorus},  let $\Theta \subset \Sigma(L)$ be the CM type corresponding to the Hodge structure $V$ \cite[Section V.C]{domain}. We may view $\Theta$ as an element of $X_*(R_{L/\mathbb{Q}}\mathbb{G}_m)$ by identifying it with $\sigma_1+\cdots +\sigma_p$ in $X_*(R_{L/\mathbb{Q}}\mathbb{G}_m)$.

Since $M$ is the smallest $\mathbb{Q}$-algebraic group such that the cocharacter $\gamma\colon \mathbb{G}_{m}\rightarrow GL(V_\mathbb{C})$ factors through $M_\mathbb{C}$, the element $\Theta$ lies in $X_*(M)$ and, in fact, the $\mathbb{Z}[\mathrm{Gal}(\overline{\mathbb{Q}}/\mathbb{Q})]$-submodule of $X_*(R_{L/\mathbb{Q}}\mathbb{G}_m)$ generated by $\Theta$ is contained in $X_*(M)$. But $X_*(H)$ consists of the elements of $X_*(M)$ which have degree $0$ in $X_*(R_{L/\mathbb{Q}}(\mathbb{G}_m))$. Thus, in particular, $\eta=\Theta-\overline{\Theta}$ is an element of $X_*(H)$. Since each embedding $\sigma\in \Sigma(L)$ has coefficient $\pm w$ in $\eta$, we have:
$$\langle\eta,\sigma\rangle=\pm w \text{ for all }\sigma\in \Sigma(L),$$
 where $\langle\ ,\ \rangle: X_*(H)\times X^*(H)\rightarrow \mathbb{Z}$ denotes the natural bilinear pairing. Then, from the description of $\chi$ in (\ref{defchi}), the integer $\langle \eta, \chi\rangle$ is the sum of $p$ terms each of which is $\pm w$. Since $p$ is odd, this means $\langle \eta, \chi\rangle$ is nonzero. Hence $\chi$ must be nonzero. So indeed $Y$ contains a nonzero element fixed under the action of $g$ and hence we have $\dim_\mathbb{Q}H=p$. This finishes Case (3) and so the statement of the theorem holds whenever $p$ is odd. 
 
So now assume $p=2$. As before, Totaro's exceptional case (6) eliminates the possibility that $L$ is a totally real quadratic field. Additionally, Totaro's exceptional case (1) eliminates the possibility that $L$ is a totally indefinite quaternion algebra over $\mathbb{C}$. Now consider the case when $L$ is of Type IV. Since $[L:\mathbb{Q}]=4$, Albert's classification yields that $L$ is a CM field of degree $2$ or $4$ over $\mathbb{Q}$. In the first case, namely when $L$ is an imaginary quadratic field, we have $m=2$ and $q=1$ and so $n_{\sigma}+n_{\overline{\sigma}}=2$ for the two embeddings $\sigma, \overline{\sigma}\in \Sigma(L)$. Hence, either  $n_{\sigma}=0$ and $n_{\overline{\sigma}}=2$ or $n_{\sigma}=n_{\overline{\sigma}}=1$. These correspond to exceptional cases (3) and (4) in Totaro's classification, so $L$ cannot be an imaginary quadratic field. We are thus left with the following possibilities for $L$:
 \begin{enumerate}
 \item $\mathbb{Q}$ (Type I)
 \item A totally definite quaternion algebra over $\mathbb{Q}$ (Type III)
\item A CM-field of degree $4$ (Type IV).
 \end{enumerate} 
 
Consider Case (1) first. In this case, the Lefschetz group of $V$ is $SO(V)$, where $SO(V)$ has rank $2$. Observe that by Lemma \ref{bound} the rank of $H$ as an algebraic group over $\mathbb{Q}$ must be greater than or equal to 2. Moreover, by Remark \ref{ss}, the group $H$ is semisimple.  Since $SO(V)$ is a $\mathbb{Q}$-form of $SO(4)$, it contains no semisimple proper subgroups of rank at least $2$.  So indeed $H$ is equal to the Lefschetz group $SO(V)$.

For Case (2), the Lefschetz group is equal to $Sp(L,^{-}),$ which is a $\mathbb{Q}$-form of $SL_2$ and hence has rank 1. Thus we must have $H=Lef(V)$.

For Case (3), the Lefschetz group is $U_L$ which is a torus of dimension $2$.  Applying Lemma \ref{bound} yields that $H$ has dimension at least $2$, so we must have $H=U_L$. Thus, indeed, when $p=2$ the Hodge group of $V$ is always equal to the Lefschetz group of $V$, which finishes the proof.
\end{proof}


\subsection{Hodge Numbers $(4,0,\ldots,0,4)$}

\begin{thm}\label{fourfold}
 Let $V$ be a simple polarizable $\mathbb{Q}$-Hodge structure of weight $w\ge1$ with Hodge numbers $(4,0,\ldots,0,4)$ with endomorphism algebra $L$.  Then the Hodge group $Hg(V)$ of $V$ is described by Table \ref{hodgetable4}. In particular, we have
 $Hg(V)=Lef(V)$
 except in the following cases:

\begin{enumerate}
\item If $L=\mathbb{Q}$ and $w$ is odd, then we can also have
$Hg(V)=SL(2)\times SO(4),$ acting on $V$ by the product of the standard representations
\item If $L=\mathbb{Q}$ and $w$ is even, then we can also have
$Hg(V)=SO(7),$
acting on $V$ by the spin representation.
\item If $L$ is an imaginary quadratic field such that $V$ is an $L$-Hodge structure with Hodge numbers $(2,0,\ldots,0,2)$, then 
$Hg(V)=R_{F/\mathbb{Q}}SU(B,^{-}),$ where $B\cong M_4(L^{\mathrm{op}})$ is the centralizer of $L$ in $\mathrm{End}_\mathbb{Q}(V)$
\item If $L$ is a CM field of degree $8$ containing an imaginary quadratic field $E$ such that $V$ is an $E$-Hodge structure with Hodge numbers $(2,0,\ldots,0,2)$, then
$Hg(V)=SU_{L/E}$.
\end{enumerate}
\end{thm}

\begin{table}[h]
        \caption{Hodge groups for $\mathbb{Q}$-Hodge structures with Hodge numbers $(4,0,\ldots,0,4)$}
        \label{hodgetable4}
        \begin{tabular}{|c|c|c|c|c|}
         \hline
 \raisebox{0pt}[15pt][0pt]{} $L$  &$[L:\mathbb{Q}]$ & \multicolumn{2}{|c|}{Possible Hodge Groups} &Equal to Lefschetz Group?\\
            \cline{3-4}
               &    & Odd Weight & Even Weight &   \\
            \hline
            \hline
            Type I&1&$Sp(8)$&$SO(8)$&Yes\\
            \cline{3-5}
            &&$SL(2)\times SO(4)$& - & No\\
            \cline{3-5}
            &&-& $SO(7)$ & No\\
            \cline{2-5}
              &$2$&$R_{F/\mathbb{Q}}Sp(_{F}V)$&$R_{F/\mathbb{Q}}SO(_{F}V)$&Yes\\
              \cline{2-5}
                &$4$&$R_{F/\mathbb{Q}}Sp(_{F}V)$&-&Yes\\
                \cline{1-5}
                  Type II&$4$&$Sp(B,^{-})$&$O^+(B,^{-})$&Yes\\
                  \cline{2-5}
                    &$8$&$R_{F/\mathbb{Q}}Sp(L,^{-})$&-&Yes\\
                    \cline{1-5}
                      Type III&$4$&$O^+(B,^{-})$&$Sp(B,^{-})$&Yes\\
                      \cline{2-5}
                        &$8$&-&$R_{F/\mathbb{Q}}Sp(L,^{-})$&Yes\\
                        \cline{1-5}
                          Type IV&$2$&$U(B,^{-})$&$U(B,^{-})$&Yes\\
                          \cline{3-5}
                          &&$SU(B,^{-})$&$SU(B,^{-})$&No\\
                          \cline{2-5}
                            &$4$&$R_{F/\mathbb{Q}}U(B,^{-})$&$R_{F/\mathbb{Q}}U(B,^{-})$&Yes\\
                            \cline{2-5}
                              &$8$&$U_L$&$U_L$&Yes\\
                              \cline{3-5}
                              &&$SU_{L/E}$&$SU_{L/E}$&No\\
            \hline
        \end{tabular}
    \end{table}

\begin{proof}
If the endomorphism algebra $L$ is of Type I, then $L$ is either $\mathbb{Q}$, a totally real quadratic field, or a totally real field of degree $4$ over $\mathbb{Q}$. However, the case when $w$ is even and $L$ is a totally real field of degree $4$ corresponds to Totaro's exceptional case (6) and thus cannot occur.  

Similarly, if $L$ is of Type II or III, then $L$ is either a quaternion algebra over $\mathbb{Q}$ or a quaternion algebra over a real quadratic field. However, Totaro's exceptional cases (1) eliminate the possibility of $L$ being a Type II (respectively Type III) quaternion algebra over a totally real quadratic field if $w$ is even (respectively odd). 

Lastly, writing $q^2=[L:F_0]$ for the degree of $L$ over its center $F_0$, if $L$ is of Type IV and $q=2$, then Albert's classification implies $m=1$. Namely, we have $[L:\mathbb{Q}]=8$ and $L$ is a central simple algebra over the imaginary quadratic field $F_0$. But, as argued in the last paragraph of the proof of Theorem \ref{2pthm}, such an $L$ is not possible. Namely, if $L$ is of Type IV, then $L$ is a CM field. 

We thus have the following list of possibilities for $L$:
\begin{enumerate}
\item $\mathbb{Q}$ (Type I)
\item A totally real quadratic field (Type I)
\item A totally real field of degree 4 (if $w$ odd) (Type I) 
\item A quaternion algebra over $\mathbb{Q}$ (Type II/ Type III)
\item A quaternion algebra over $F$ with $[F:\mathbb{Q}]=2$ (Type II if $w$ odd/ Type III if $w$ even)
\item An imaginary quadratic field (Type IV)
\item A CM field of degree 4 (Type IV)
\item A CM field of degree 8 (Type IV).
\end{enumerate}

Using Table \ref{exphg}, when $L$ is of Type I, then the Lefschetz group of $V$ is $R_{F/\mathbb{Q}}Sp(_{F}V)$ when $w$ is odd and $R_{F/\mathbb{Q}}SO(_{F}V)$ when $w$ is even. Case (1) is then taken care of by Proposition \ref{four1}, Case (2) is taken care of by Proposition \ref{2TypeI}, and Case (3) is taken care of by Proposition \ref{rational}. From these, we conclude that when $L$ is of Type I, the Hodge group of $V$ is always equal to the Lefschetz group of $V$, except when $L=\mathbb{Q}$, in which case, the two additional groups $SL(2)\times SO(4)$ in the odd weight case, acting by the product of the standard representations, and $SO(7)$ in the even-weight case, acting by the spin representation, are also possible. 

Similarly, Case (4) is taken care of by Proposition \ref{2quat} as well as by Proposition \ref{4quat} and Case (5) is taken care of by Proposition \ref{oddquat}. From these, we conclude that when $L$ is of Type II or III, then the Hodge group of $V$ is always equal to the Lefschetz group of $V$. This leaves only the cases when $L$ is of Type IV to consider. The result of the theorem then follows from the following propositions below: Proposition \ref{ivimag}, Proposition \ref{ivdeg4}, and Proposition \ref{ivdeg8}, which address Cases (6), (7), and (8) respectively.

\begin{prop}\label{ivimag} Let $V$ be a simple polarizable $\mathbb{Q}$-Hodge structure of weight $w\ge1$ and Hodge numbers $(4,0,\ldots,0,4)$ such that the endomorphism algebra $L$ of $V$ is an imaginary quadratic field. Let $B\cong M_4(L^{\mathrm{op}})$ be the centralizer of $L$ in $\mathrm{End}_\mathbb{Q}(V)$,  let $\Sigma(L)=\{\sigma, \overline{\sigma}\}$ be the set of embeddings of $L$ into $\mathbb{C}$, and  let $n_\sigma=\dim V^{w,0}(\sigma)$ and $n_{\overline{\sigma}}=\dim V^{w,0}(\overline{\sigma})$.  Then, either $\{n_{\sigma},n_{\overline{\sigma}}\}=\{1,3\}$ or $\{n_{\sigma},n_{\overline{\sigma}}\}=\{2,2\}$ and we have:
\begin{equation*}
Hg(V)=
\begin{cases}
U(B,^{-}) &\mbox{if } \{n_{\sigma},n_{\overline{\sigma}}\}=\{1,3\}\\
SU(B,^{-}) &\mbox {if } \{n_{\sigma},n_{\overline{\sigma}}\}=\{2,2\}.
\end{cases}
\end{equation*}
\end{prop}

\begin{proof}
Since $L$ is an imaginary quadratic field, we know $m=4$ and $q=1$, so $n_\sigma+n_{\overline{\sigma}}=4$. Totaro's exceptional case (3) implies that we cannot have $\{n_\sigma,n_{\overline{\sigma}}\}=\{0,4\}$. Therefore either we have $\{n_{\sigma},n_{\overline{\sigma}}\}=\{1,3\}$ or we have $\{n_{\sigma},n_{\overline{\sigma}}\}=\{2,2\}$. The case when $\{n_{\sigma},n_{\overline{\sigma}}\}=\{1,3\}$ is taken care of by Proposition \ref{imaginary}. In the case when $\{n_{\sigma},n_{\overline{\sigma}}\}=\{2,2\}$, then $V$ is an $L$-Hodge structure with Hodge numbers $(2,0,\ldots,0,2)$. Hence the result follows from Proposition \ref{Ealg}.
\end{proof}
\begin{prop} \label{ivdeg4}Let $V$ be a simple polarizable $\mathbb{Q}$-Hodge structure of weight $w\ge1$ and Hodge numbers $(4,0,\ldots,0,4)$ such that the endomorphism algebra $L$ of $V$ is a CM field of degree 4.  Let $F$ be  maximal totally real subfield of $L$ and let $B\cong M_2(L^{\mathrm{op}})$ be the centralizer of $L$ in $\mathrm{End}_\mathbb{Q}(V)$. Then  $Hg(V)=R_{F/\mathbb{Q}}U(B,^{-}).$
\end{prop}

\begin{proof}
The following proof borrows from Moonen and Zarhin's proof \cite[7.5]{fourfold} of the analogous result for simple abelian fourfolds.  

Consider the set of embeddings $\Sigma(L)=\{\sigma_1,\overline{\sigma}_1,\sigma_2,\overline{\sigma}_2\}$ of the CM field $L$ into $\mathbb{C}$. Then, since $[L:\mathbb{Q}]=4$, we know $m=2$ and $q=1$, hence for each $i\in \{1,2\}$, we have $n_{\sigma_i} + n_{\overline{\sigma}_i}=2$.  However, the case $n_{\sigma_1}n_{\overline{\sigma}_1}+n_{\sigma_2}n_{\overline{\sigma}_2}=0$ corresponds to exceptional case (3) in Totaro's classification and the case $n_{\sigma_1}=n_{\sigma_2}=1$ corresponds to exceptional case (4). So neither of these can occur. Thus, without loss of generality, we have 
\begin{equation}\label{dimension}
(n_{\sigma_1},n_{\overline{\sigma}_1})=(2,0) \text{ and } (n_{\sigma_2},n_{\overline{\sigma}_2})=(1,1).
\end{equation}

Let $Z$ denote the center of $H$. Then since $H$ is contained in the Lefschetz group $R_{F/\mathbb{Q}}U(B,^{-})$, we must have $Z$ contained in the $2$-dimensional torus $U_L:=\ker( \mathrm{Norm}_{L}\colon R_{L/\mathbb{Q}}\mathbb{G}_m\rightarrow R_{F/\mathbb{Q}}\mathbb{G}_m)$. 

Suppose $Z$ is trivial. This implies $H\subseteq R_{F/\mathbb{Q}}SU(B,^{-}).$ However the centralizer of $R_{F/\mathbb{Q}}SU(B,^{-})$ in $\mathrm{End}_\mathbb{Q}(V)$ is a quaternion algebra over $F$, whereas the centralizer of $H$ in $\mathrm{End}_\mathbb{Q}(V)$ is $L$. Hence $Z$ must be nontrivial.

 Suppose $Z$ has dimension $1$. Then by Lemma 7.3 in \cite{fourfold} there exists an imaginary quadratic subfield $E$ of $L$ such that $Z=SU_{L/E}$. Now let $C\cong M_4(E^{\mathrm{op}})$ be the centralizer of $E$ in $\mathrm{End}_\mathbb{Q}(V)$. Then having the center $Z$ of $H$ equal to $SU_{L/E}$ implies 
$H\subseteq SU(C,^{-}).$  Let $\nu$ and $\overline{\nu}$ be the two embeddings of $E$ into $\mathbb{C}$. Then letting $n_\nu=\dim V^{w,0}(\nu)$, we have $n_\nu+n_{\overline{\nu}}=4$. By Lemma \ref{sunitary}, we must have $n_\nu=2=n_{\overline{\nu}}$. However, since $E$ is contained in $L$, the equalities in (\ref{dimension}) imply $\{n_\nu,n_{\overline{\nu}}\}=\{1,3\}$, which is a contradiction. Hence $Z$ cannot have dimension $1$ and so we have shown that $Z$ is $2$-dimensional, meaning $Z=U_L$. 

Now let $\rho_1$ and $\rho_2$ be the two embeddings of $F$ into $\mathbb{C}$.  Without loss of generality, we may assume that the two embeddings $\sigma_1,\overline{\sigma}_1\in \Sigma(L)$ both extend $\rho_1$ and that the two embeddings $\sigma_2,\overline{\sigma}_2 \in \Sigma(L)$ both extend $\rho_2$.

The action of $F\otimes _\mathbb{Q} \mathbb{C}$ on $V\otimes _\mathbb{Q}\mathbb{C}$ yields a decomposition
$$V\otimes _{\mathbb{Q}}\mathbb{C}=X_{\rho_1}\oplus X_{\rho_2},$$
where $X_{\rho_1}$ and $X_{\rho_2}$ are $4$-dimensional $\mathbb{C}$-vector spaces. 

For a fixed polarization $\langle,\rangle$ of $V$ let $\psi:V\times V\rightarrow F$ be the bilinear form such that
$\langle,\rangle=\mathrm{Tr}^F_\mathbb{Q}\circ \psi$ and let $\psi_{\rho_1}$ and $\psi_{\rho_2}$ be the restrictions of $\psi$ to to $X_{\rho_1}$ and $X_{\rho_2}$ respectively. For any $v,w\in V$ and any $f\in L$ we have
 $\psi(fv,w)=\psi(v,\overline{f}w)$. Hence, since the field $F$ is fixed under the Rosati involution on $L$, the vector spaces $X_{\rho_1}$ and $X_{\rho_2}$ are orthogonal with respect to $\psi$. So $\psi_{\rho_1}$ and $\psi_{\rho_2}$ are nondegenerate alternating bilinear forms such that
\begin{equation}\label{unitarydecomp}H_\mathbb{C}\subseteq U(X_{\rho_1}, \psi_{\rho_1})\oplus U(X_{\rho_2}, \psi_{\rho_2}).\end{equation}
Hence, if $H_\mathbb{C}^{ss}$ denotes the semisimple part of $H_\mathbb{C}$ we have
\begin{equation}\label{sspart}H_\mathbb{C}^{ss}\subseteq SU(X_{\rho_1}, \psi_{\rho_1})\oplus SU(X_{\rho_2}, \psi_{\rho_2}).\end{equation}
For $i\in \{1,2\}$ write
\begin{equation}\label{rhodecomp}X_{\rho_i}=V(\sigma_i)\oplus V(\overline{\sigma_i}),\end{equation}
where $V(\sigma_i)$ and $V(\overline{\sigma_i})$ are $2$-dimensional irreducible $H_\mathbb{C}$-modules.

Using the decompositions (\ref{unitarydecomp}) and (\ref{rhodecomp}) in combination, we may write the center $Z_\mathbb{C}$ of $H_\mathbb{C}$ as
$$Z_\mathbb{C}=\{(z_1\cdot \mathrm{Id}, -z_1\cdot \mathrm{Id}, z_2\cdot \mathrm{Id}, -z_2\cdot \mathrm{Id})\mid z_1,z_2\in \mathbb{C}\} \subset U(X_{\rho_1}, \psi_{\rho_1})\oplus U(X_{\rho_2}, \psi_{\rho_2}).$$

Because the $V(\sigma_i)$ and $V(\overline{\sigma_i})$ are $2$-dimensional irreducible $H_\mathbb{C}$-modules, the above description of $Z_\mathbb{C}$ yields that the projection of $H_\mathbb{C}^{ss}$ onto each factor $SU(X_{\rho_i}, \psi_{\rho_i})$ must be nonzero. But both of the $SU(X_{\rho_i}, \psi_{\rho_i})$ factors are simple. Thus, $H_\mathbb{C}^{ss}$ surjects onto each factor $SU(X_{\rho_i}, \psi_{\rho_i})$ in the inclusion in (\ref{sspart}). 

The argument in the proof of Lemma \ref{goursat} then yields that the inclusion in (\ref{unitarydecomp}) is in fact an equality. Since we already showed that the center of $H_\mathbb{C}$ is all of $U_L$, we have thus shown that $H$ is equal to the Lefschetz group $R_{F/\mathbb{Q}}U(B,^{-})$.
\end{proof}
The last remaining case to deal with in the proof of Theorem \ref{fourfold} is Case (8), namely the case when the endomorphism algebra is a CM field of degree $8$.

\begin{prop}\label{ivdeg8} Let $V$ be a simple polarizable $\mathbb{Q}$-Hodge structure of weight $w\ge1$ and Hodge numbers $(4,0,\ldots,0,4)$ with  endomorphism algebra $L$ a CM field of degree 8. Letting $F$ be the maximal totally real subfield of $L$, we have
\begin{equation*}
Hg(V)=
\begin{cases}
SU_{L/E} &
\mbox{if } L \text{ contains an imaginary quadratic field }E \text{ such that }
V \text{ is an }\\
&\quad E\text{-Hodge structure with Hodge numbers }(2,0,\ldots,0,2)\\
 U_L&\mbox{otherwise.}
\end{cases}
\end{equation*} 
\end{prop}
\begin{proof}
Let $H=Hg(V)$ be the Hodge group of $V$. By Table \ref{exphg}, the Lefschetz group of $V$ is the $4$-dimensional torus $U_L$. By Lemma \ref{bound}, the rank of $H$ is greater than or equal to $\mathrm{log}_2(8)=3$. So either $H$ is a $3$-dimensional subtorus of $U_L$ or $H$ is all of $U_L$. 

Suppose $L$ contains an imaginary quadratic field $E$ such that $V$ is an $E$-Hodge structure with Hodge numbers $(2,0,\ldots,0,2)$. Then by Proposition \ref{contained}, the group $H$ is contained in the $3$-dimensional torus $SU_{L/E}$. Hence in this case $H=SU_{L/E}$. 

Conversely, if $H$ is $3$-dimensional, then by Lemma 7.3 in \cite{fourfold}, there exists an imaginary quadratic field $E$ in $L$ such that $H=SU_{L/E}$.   Moreover, by Lemma \ref{sunitary} this means $V$ must be an $E$-Hodge structure with Hodge numbers $(2,0,\ldots,0,2)$.

Hence, if $L$ contains no such field $E$, then $H$ must be $4$-dimensional and hence $H$ must be equal to the Lefschetz group $U_L$.
\end{proof}

Proposition \ref{ivimag}, Proposition \ref{ivdeg4}, and Proposition \ref{ivdeg8} thus indeed verify the statement of Theorem \ref{fourfold} in Cases (6), (7), and (8) respectively. Since we have previously confirmed the statement of Theorem \ref{fourfold} in Cases (1)-(5), this completes the proof of Theorem \ref{fourfold}.

\end{proof}


\subsection{Hodge Numbers $(2p,0,\ldots,0,2p)$}

\begin{thm}\label{2phodge}
Let $V$ be a simple polarizable $\mathbb{Q}$-Hodge structure of weight $w\ge1$ with Hodge numbers $(2p,0,\ldots, 0, 2p)$, where $p$ is an odd prime. If the endomorphism algebra $L$ of $V$ is of Type I, II, or III, then 
\[Hg(V)=Lef(V).\]

However, if $L$ is of Type IV and $L$ contains an imaginary quadratic field $E$ such that $V$ is an $E$-Hodge structure with Hodge numbers $(p,0,\ldots,0,p)$, then
\begin{equation*}
Hg(V)=
\begin{cases}
R_{F/\mathbb{Q}}SU(B,^-) &\mbox{if } [L:\mathbb{Q}]\ne 4p\\
SU_{L/E} &\mbox{if } [L:\mathbb{Q}]=4p \text{ and  } L/\mathbb{Q} \text{ is a Galois extension.}
\end{cases}
\end{equation*}
\end{thm}
\begin{proof}
If $L$ is of Type I, II, or III, we use Albert's and Totaro's classifications to obtain a list of possibilities for $L$.  As in previous proofs, the case when $w$ is even and $L$ is a totally real field of degree $2p$ corresponds to Totaro's exceptional case (6) and thus cannot occur.  Additionally, Totaro's exceptional cases (1) eliminate the possibility of $L$ being a Type II (respectively Type III) quaternion algebra over a totally real field of degree $p$ over $\mathbb{Q}$ if $w$ is even (respectively odd). 

If $L$ is of Type IV and satisfies the assumptions of the theorem, then by Theorem \ref{2pthm}, we know that $L$ must be a CM field. 

We thus have the following list of possibilities for $L$:
\begin{enumerate}
\item $\mathbb{Q}$ (Type I)
\item A real quadratic field (Type I)
\item A totally real field of degree $p$ (Type I)
\item A totally real field of degree $2p$ (if $w$ odd) (Type I)
\item A quaternion algebra over $\mathbb{Q}$ (Type II/ Type III)
\item A quaternion algebra over $F$, where $[F:\mathbb{Q}]=p$ (Type II if $w$ odd/ Type III if $w$ even)
\item A CM field of degree $2$ (Type IV)
\item A CM field of degree $4$ (Type IV)
\item A CM field of degree $2p$ (Type IV)
\item A CM field of degree $4p$ (Type IV).
\end{enumerate}

First consider the cases when $L$ is of Type I. Then by Table \ref{exphg}, the Lefschetz group of $V$ is $R_{F/\mathbb{Q}}Sp(_{F}V)$, when $w$ is odd, and $R_{F/\mathbb{Q}}SO(_{F}V)$, when $w$ is even. So consider Case (1). By Proposition \ref{twiceodd}, the Hodge group $H=Hg(V)$ is either equal to $Lef(V)$ or, when $w$ is even, we may also have $R_{L/\mathbb{Q}}SU(2^k)$, where $4p={2^k\choose 2^{k-1}}$ for some $k\ge 3$. However, by a combinatorial argument similar to the one used in the proof of Theorem \ref{primethm}, this latter case is impossible. 

In Case (2), Proposition \ref{rational} yields that $H$ is either equal to $Lef(V)$ or, when $w$ is even, we may also have $R_{L/\mathbb{Q}}SU(2^k)$, where $2p={2^k\choose 2^{k-1}}$ for some $k\ge 3$. However, as verified in the proof of Theorem \ref{primethm}, this latter case is impossible. Case (3) is taken care of by Proposition \ref{2TypeI}. Case (4) is taken care of by Proposition \ref{rational} since in this case $\frac{n}{[L:\mathbb{Q}]}=1$ and thus in both even and odd weights the only possibility for $H$ is $Lef(V)$. This finishes the cases for $L$ of Type I. 

For $L$ of Type II or Type III, referring to Table \ref{exphg}, Case (5) is taken care of by Proposition \ref{oddquat}, again using that $2p$ cannot be of the form ${2^k\choose 2^{k-1}}$ for $k\ge 3$. Case (6) is also taken care of by Proposition \ref{oddquat}.

When $L$ is of Type IV, Table \ref{exphg} yields that the Lefschetz group is $R_{F/\mathbb{Q}}U(B,^{-})$. Namely, under the hypotheses of the theorem, the predicted Hodge group in the cases when $L$ is of Type IV is strictly smaller than the Lefschetz group. The statement of the theorem for these Type IV endomorphism algebra cases, meaning Cases (7) through (10), follows from Theorem \ref{2pthm}.
\end{proof}

\section{Applications to the Hodge Conjecture for Abelian Varieties}\label{hgconimp}

In Section \ref{mainresults}, we determined the possible Hodge groups of simple polarizable $\mathbb{Q}$-Hodge structures with Hodge numbers $(n,0,\ldots,0,n)$ when $n$ was equal to $1$, a prime $p$, $4$ and $2p$. The results for $n=1$, $p$, and $4$ are generalizations of previous results in  \cite{ribet1}, \cite{tankeevrib}, and  \cite{fourfold} about the possible Hodge groups of simple $n$-dimensional abelian varieties. However, the results in Section \ref{mainresults} about the Hodge groups when $n$ is equal $2p$ are new. 

Since, by Remark \ref{equiv}, there is a polarization-preserving equivalence of categories between the category of $\mathbb{Q}$-Hodge structures of  odd weight and Hodge numbers $(n,0\ldots,0,n)$, and the category of complex abelian varieties of dimension $n$, it is natural to ask about the implications of Theorem \ref{2phodge} for complex abelian varieties. In particular, it is natural to wonder about the implications in terms of both the Hodge Conjecture and the General Hodge conjecture for these simple complex abelian varieties of dimension $2p$. 

In order to simplify notation, in the case of an abelian variety $A$, we will denote by $Hg(A)$ and $Lef(A)$ the Hodge and Lefschetz groups respectively of the $\mathbb{Q}$-Hodge structure $V=H^1(A,\mathbb{Q})$. 

If $A$ has dimension $n$, we also introduce the notation $W(A)$ to denote the set of CM fields $E$ such that $V$ is an $E$-Hodge structure with Hodge numbers $(\frac{n}{[E:\mathbb{Q}]},0,\ldots,0,\frac{n}{[E:\mathbb{Q}]})$.


\begin{cor}\label{cor1}Let $A$ be a simple complex abelian variety of dimension $2p$, where $p$ is an odd prime. Suppose the endomorphism algebra $L$ of the $\mathbb{Q}$-Hodge structure $V=H^1(A,\mathbb{Q})$ satisfies either
\begin{enumerate}
\item $L$ is of Type I, II, or III
\item $L$ is of Type IV, there exists an imaginary quadratic field $E\in W(A)$, and $[L:\mathbb{Q}]\ne 4p$.
\end{enumerate}
Then,  if the Hodge conjecture is true for all powers of $A$, then the General Hodge Conjecture is true for all powers of $A$. 
\end{cor}

\begin{proof}
By Theorem \ref{2phodge}, the Hodge group $Hg(A)$ is semisimple and is equal to the semisimple part $Lef(A)^{ss}$ of the Lefschetz group of $A$. Moreover, when $L$ is of Type III, then by Shimura's classification of the possible endomorphism algebras of a simple abelian variety \cite[Theorem 5]{shimura}, we must have $L$ a quaternion algebra over $\mathbb{Q}$. Namely $\frac{4p}{[L:\mathbb{Q}]}$ is equal to $p$, where $p$ is odd. So $A$ satisfies the following two conditions:
\begin{enumerate}
\item $Hg(A)=Lef(A)^{ss}$
\item If $L$ is of Type III, then $\frac{2\dim A}{[L:\mathbb{Q}]}$ is odd.
\end{enumerate}
These conditions are exactly the hypotheses of a result of Abdulali \cite[Theorem 5.1]{abdulali1}, which then shows that under the above circumstances, the Hodge Conjecture for all powers of $A$ implies the General Hodge Conjecture for all powers of $A$. 
\end{proof}


For any $E\in W(A)$, let $J$ be the maximal totally real subfield of $E$ and let $C$ be the the centralizer of $E$ in $\mathrm{End}_\mathbb{Q}(V)$. The Hodge structure $V$ may be viewed as an $E$-vector space, say of dimension $l$.  Thus define
$$W_E=\bigwedge^{l}_E V.$$
Since $V$ is an $E$-Hodge structure, a result of Moonen and Zarhin \cite[Section 6]{mz2} shows that $W_E$ consists entirely of Hodge classes. The elements of $W_E$ are called \emph{Weil classes}.

We now introduce the group $M(A)$, originally defined by Murty in \cite[Section 3.6.4]{murty}.  
\begin{defn} For a simple complex abelian variety $A$, let the \emph{Murty group} $M(A)$ be given by
$$M(A)=Lef(A)\bigcap \left( \bigcap_{E\in W(A)} R_{J/\mathbb{Q}}SU(C,^{-})\right).$$
\end{defn}

Thus for any simple complex abelian variety $A$, we have:
\begin{equation}\label{mcontain}Hg(A)\subseteq M(A)\subseteq Lef(A).\end{equation}
In \cite{murty}, Murty proves the following property about the Murty group:
\begin{prop}\cite[Proposition 3.8]{murty}\label{murtyprop} For a complex abelian variety $A$, the Hodge ring
$$\mathcal{B}^{\bullet}(A^k)=\bigoplus_{l\ge0} \left(H^{2l}(A^k,\mathbb{Q})\cap H^{l,l}\right)$$
is generated by divisors and Weil classes for all $k\ge1$ if and only if 
$Hg(A)=M(A).$
\end{prop}


\begin{cor}\label{hgcon} Let $A$ be a simple complex abelian variety of dimension $2p$, where $p$ is an odd prime. Suppose the endomorphism algebra $L$ of the $\mathbb{Q}$-Hodge structure $V=H^1(A,\mathbb{Q})$ satisfies either
\begin{enumerate}
\item $L$ is of Type I, II, or III
\item $L$ is of Type IV, there exists an imaginary quadratic field $E\in W(A)$, and if $[L:\mathbb{Q}]=4p$, then $L/\mathbb{Q}$ is Galois.
\end{enumerate}
Then for every $k\ge 1$, the Hodge ring $\mathcal{B}^{\bullet}(A^k)=\bigoplus_{l\ge0} \left(H^{2l}(A^k,\mathbb{Q})\cap H^{l,l}\right)$ is generated by divisors and Weil classes. \end{cor} 

\begin{proof}
When $L$ is of Type I, II, or III in Albert's classification, then by \ref{murtyprop} we have
$Lef(A)=M(A).$
Since $A$ corresponds to a simple polarizable $\mathbb{Q}$-Hodge structure $V$ with Hodge numbers $(2p,2p)$,  when $L$ is of Type I, II, or III, then  by Theorem \ref{2phodge}, we have
$Hg(A)=Lef(A).$
Hence, in Case (1) in the statement of the corollary, the Hodge group of $A$ is indeed equal to the Murty group of $A$.

In the situation of Case (2), by Theorem \ref{2phodge} we have
\begin{equation*}
 Hg(A)=
 \begin{cases}
 R_{F/\mathbb{Q}}SU(B,^-) &\mbox{if }[L:\mathbb{Q}]\ne 4p\\
SU_{L/E} &\mbox{if } [L:\mathbb{Q}]=4p,
\end{cases}
\end{equation*}
where $B$ is the centralizer of $L$ in $\mathrm{End}_\mathbb{Q}(V)$. Moreover, in this case, the Lefschetz group of $A$ is:
\begin{equation*}
 Lef(A)=
 \begin{cases}
 R_{F/\mathbb{Q}}U(B,^-) &\mbox{if }[L:\mathbb{Q}]\ne 4p\\
U_{L/E} &\mbox{if } [L:\mathbb{Q}]=4p,
\end{cases}
\end{equation*}
However, observe that this means for some $E\in W(A)$ with totally real subfield $J$ and centralizer $C$ in $\mathrm{End}_\mathbb{Q}(V)$, we have:
$$Hg(A)=Lef(A)\cap R_{J/\mathbb{Q}}SU(C,^{-}).$$
Namely,
$M(A) \subseteq Hg(A)$
and so by (\ref{mcontain}) above, we have $M(A)=Hg(A)$. Then Murty's result, Proposition \ref{murtyprop}, implies that for every $k\ge 1$ the Hodge ring $\mathcal{B}^{\bullet}(A^k)$ is generated by divisors and Weil classes.
\end{proof}


\begin{cor}\label{genhg}
Let $A$ be a simple abelian variety of dimension $2p$, where $p$ is an odd prime. Suppose the endomorphism algebra $L$ of the corresponding Hodge structure $V=H^1(X,\mathbb{Q})$ is of Type I or II in Albert's classification. Then both the Hodge and General Hodge Conjectures are satisfied for every power of $A$.
\end{cor} 

\begin{proof}
In \cite[Section 13]{mz2}, Moonen and Zarhin show that both the Type III and Type IV cases in the statement of Corollary \ref{hgcon} will yield exceptional Hodge classes in $W_E$, but that this will not occur in the Type I and Type II cases. Namely if $L$ is of Type I or Type II, then all of the Weil classes in $W_E$ are actually just divisor classes. Hence the Hodge ring $\mathcal{B}^{\bullet}(A^k)$ is generated by divisors and so the Hodge Conjecture is satisfied for every power of $A$. However, by Corollary \ref{cor1}, since the Hodge Conjecture holds for all powers of $A$, the General Hodge Conjecture holds for all powers of $A$.
 \end{proof}


\normalsize{\bibliography{HgGroups14.bib}}

\begin{thebibliography}{KMRT98}

\bibitem[Abd97]{abdulali1}
Salman Abdulali.
\newblock Abelian varieties and the general {H}odge conjecture.
\newblock {\em Compositio Math.}, 109(3):341--355, 1997.

\bibitem[Alb39]{albert}
A.~Adrian Albert.
\newblock {\em Structure of {A}lgebras}.
\newblock American Mathematical Society Colloquium Publications, vol. 24.
  American Mathematical Society, New York, 1939.

\bibitem[Ara16]{arapura}
Donu Arapura.
\newblock Geometric {H}odge structures with prescribed {H}odge numbers.
\newblock In {\em Recent advances in {H}odge theory}, volume 427 of {\em London
  Math. Soc. Lecture Note Ser.}, pages 414--421. Cambridge Univ. Press,
  Cambridge, 2016.

\bibitem[BMM16]{bergeron}
Nicolas Bergeron, John Millson, and Colette Moeglin.
\newblock The {H}odge conjecture and arithmetic quotients of complex balls.
\newblock {\em Acta Math.}, 216(1):1--125, 2016.

\bibitem[Bor12]{borel}
Armand Borel.
\newblock {\em Linear algebraic groups}, volume 126.
\newblock Springer, 2012.

\bibitem[Bou75]{bourbaki}
Nicolas Bourbaki.
\newblock {\em El{\'e}ments de math{\'e}matique: Groupes et alg{\`e}bres de
  Lie: Chapitre 7, Sous-alg{\`e}bres de Cartan, {\'e}l{\'e}ments r{\'e}guliers.
  Chapitre 8, Alg{\`e}bres de Lie semi-simples d{\'e}ploy{\'e}es}.
\newblock Hermann, 1975.

\bibitem[Dod87]{dodson}
B.~Dodson.
\newblock On the {M}umford-{T}ate group of an abelian variety with complex
  multiplication.
\newblock {\em J. Algebra}, 111(1):49--73, 1987.

\bibitem[FH04]{fulton}
William Fulton and Joe Harris.
\newblock {\em Representation Theory: A First Course}.
\newblock Springer, 2004.

\bibitem[GGK12]{domain}
Mark Green, Phillip~A Griffiths, and Matt Kerr.
\newblock {\em Mumford-Tate Groups and Domains: Their Geometry and Arithmetic
  (AM-183)}.
\newblock Princeton University Press, 2012.

\bibitem[KMRT98]{boi}
Max-Albert Knus, Alexander Merkurjev, Markus Rost, and Jean-Pierre Tignol.
\newblock {\em The book of involutions}.
\newblock American Mathematical Society, 1998.

\bibitem[Kub65]{kubota}
Tomio Kubota.
\newblock On the field extension by complex multiplication.
\newblock {\em Trans. Amer. Math. Soc.}, 118:113--122, 1965.

\bibitem[Moo99]{moonenmt}
Ben Moonen.
\newblock Notes on {M}umford-{T}ate groups.
\newblock \url{http://www.math.ru.nl/~bmoonen/Lecturenotes/CEBnotesMT.pdf},
  1999.

\bibitem[Mum69]{mumford}
D.~Mumford.
\newblock A note of {S}himura's paper ``{D}iscontinuous groups and abelian
  varieties''.
\newblock {\em Math. Ann.}, 181:345--351, 1969.

\bibitem[Mur84]{murty2}
V.~Kumar Murty.
\newblock Exceptional {H}odge classes on certain abelian varieties.
\newblock {\em Math. Ann.}, 268(2):197--206, 1984.

\bibitem[Mur00]{murty}
V.~Kumar Murty.
\newblock Hodge and {W}eil classes on abelian varieties.
\newblock In {\em The arithmetic and geometry of algebraic cycles ({B}anff,
  {AB}, 1998)}, volume 548 of {\em NATO Sci. Ser. C Math. Phys. Sci.}, pages
  83--115. Kluwer Acad. Publ., Dordrecht, 2000.

\bibitem[MZ98]{mz2}
B.~J.~J. Moonen and Yu.~G. Zarhin.
\newblock Weil classes on abelian varieties.
\newblock {\em J. Reine Angew. Math.}, 496:83--92, 1998.

\bibitem[MZ99]{fourfold}
B.~J.~J. Moonen and Yu.~G. Zarhin.
\newblock Hodge classes on abelian varieties of low dimension.
\newblock {\em Math. Ann.}, 315(4):711--733, 1999.

\bibitem[Orr15]{orr}
Martin Orr.
\newblock Lower bounds for ranks of {M}umford-{T}ate groups.
\newblock {\em Bull. Soc. Math. France}, 143(2):229--246, 2015.

\bibitem[Poh68]{pohlmann}
Henry Pohlmann.
\newblock Algebraic cycles on abelian varieties of complex multiplication type.
\newblock {\em Ann. of Math. (2)}, 88:161--180, 1968.

\bibitem[Rib83]{ribet1}
Kenneth~A. Ribet.
\newblock Hodge classes on certain types of abelian varieties.
\newblock {\em Amer. J. Math.}, 105(2):523--538, 1983.

\bibitem[Rib81]{ribet2}
K.~A. Ribet.
\newblock Division fields of abelian varieties with complex multiplication.
\newblock {\em M{\'e}m. Soc. Math. France (N.S.)}, (2):75--94, 1980/81.
\newblock Abelian functions and transcendental numbers (Colloq., {{\'E}}tole
  Polytech., Palaiseau, 1979).

\bibitem[RS62]{primecounting}
J.~Barkley Rosser and Lowell Schoenfeld.
\newblock Approximate formulas for some functions of prime numbers.
\newblock {\em Illinois J. Math.}, 6:64--94, 1962.

\bibitem[Sch15]{schreieder}
Stefan Schreieder.
\newblock On the construction problem for {H}odge numbers.
\newblock {\em Geom. Topol.}, 19(1):295--342, 2015.

\bibitem[Ser67]{serre}
J.-P. Serre.
\newblock Sur les groupes de {G}alois attach{\'e}s aux groupes
  {$p$}-divisibles.
\newblock In {\em Proc. {C}onf. {L}ocal {F}ields ({D}riebergen, 1966)}, pages
  118--131. Springer, Berlin, 1967.

\bibitem[Shi63]{shimura}
Goro Shimura.
\newblock On analytic families of polarized abelian varieties and automorphic
  functions.
\newblock {\em Ann. of Math. (2)}, 78:149--192, 1963.

\bibitem[Tan82]{tankeevrib}
S.~G. Tankeev.
\newblock Cycles on simple abelian varieties of prime dimension.
\newblock {\em Izv. Akad. Nauk SSSR Ser. Mat.}, 46(1):155--170, 192, 1982.

\bibitem[Tan01]{tankeev}
S.~G. Tankeev.
\newblock Cycles of small codimension on a simple abelian variety.
\newblock {\em J. Math. Sci. (New York)}, 106(5):3365--3382, 2001.
\newblock Algebraic geometry, 11.

\bibitem[Tot15]{totaro}
Burt Totaro.
\newblock Hodge structures of type {$(n,0,\ldots,0,n)$}.
\newblock {\em Int. Math. Res. Not. IMRN}, (12):4097--4120, 2015.

\bibitem[Voi02]{voisin}
Claire Voisin.
\newblock {\em Hodge theory and complex algebraic geometry. {I}}, volume~76 of
  {\em Cambridge Studies in Advanced Mathemetics}.
\newblock Cambridge University Press, 2002.

\bibitem[Wei79]{weil}
Andr{{\'e}} Weil.
\newblock Abelian varieties and the {H}odge ring.
\newblock In {\em Collected Papers.}, volume III, [1977c], pages 421--429.
  Springer-Verlag, New York, 1979.

\end{thebibliography}
\bibliographystyle{alpha}

\end{document}